\documentclass[11pt,times,letter]{article}

\usepackage{amsmath, amsfonts,amsthm,amssymb,bm}
\usepackage{dsfont}

\usepackage[english]{babel}
\usepackage{float}

\usepackage[normalem]{ulem}

\usepackage{subfigure}
\usepackage{latexsym,amssymb,amsfonts,graphicx}
\usepackage{amsmath,dsfont}
\usepackage{verbatim}
\usepackage{mathrsfs}
\usepackage{bm}
\usepackage{color}
\usepackage{epsfig}
\usepackage{epstopdf}
\usepackage[title]{appendix}
\usepackage{url}
\usepackage{bm}
\usepackage{natbib}
\usepackage{booktabs}
\usepackage[colorlinks,linkcolor=blue,citecolor=blue,anchorcolor=blue]{hyperref}

\usepackage{algorithm}
\usepackage{algpseudocode}

\newcommand{\Px}{ \mathbb{P} }

\newcommand{\Ex}{ \mathbb{E} }

\def\esssup_#1{\underset{#1}{\mathrm{ess\,sup\, }}}
\def\essinf_#1{\underset{#1}{\mathrm{ess\,inf\, }}}
\def\argmax_#1{\underset{#1}{\mathrm{arg\,max\, }}}
\def\argmin_#1{\underset{#1}{\mathrm{arg\,min\, }}}

\newcommand{\F}{\mathcal{F}}

\newcommand{\R}{\mathds{R}}

\newtheorem{theorem}{Theorem}[section]
\newtheorem{definition}{Definition}[section]
\numberwithin{equation}{section}
\newtheorem{assumption}[theorem]{Assumption}
\newtheorem{proposition}[theorem]{Proposition}
\newtheorem{remark}[theorem]{Remark}
\newtheorem{lemma}[theorem]{Lemma}

\allowdisplaybreaks

\definecolor{Red}{rgb}{1.00, 0.00, 0.00}

\definecolor{DRed}{rgb}{0.5, 0.00, 0.00}

\definecolor{Blue}{rgb}{0.00, 0.00, 1.00}

\definecolor{Green}{rgb}{0.0, 0.4, 0.0}

\title{Continuous-time reinforcement learning for optimal\\switching over multiple regimes}

\author{ 
Yijie Huang \thanks{Department of Applied Mathematics, The Hong Kong Polytechnic University, Kowloon, Hong Kong. Email:\url{yijie.huang@polyu.edu.hk}}
\and Mengge Li\thanks{Department of Applied Mathematics,  The Hong Kong Polytechnic University, Kowloon, Hong Kong. Email:\url{meng-ge.li@polyu.edu.hk}}
\and Xiang Yu\thanks{Department of Applied Mathematics,  The Hong Kong Polytechnic University, Kowloon, Hong Kong. Email:\url{xiang.yu@polyu.edu.hk}}
	\and Zhou Zhou\thanks{School of Mathematics and Statistics, University of Sydney, Sydney, Australia. Email:\url{zhou.zhou@sydney.edu.au}} }

\begin{document}
\date{\vspace{-1cm}}
\maketitle

\begin{abstract}
This paper studies the continuous-time reinforcement learning (RL) for optimal switching problems across multiple regimes. We consider a type of exploratory formulation under entropy regularization where the agent randomizes both the timing of switches and the selection of regimes through the generator matrix of an associated continuous-time finite-state Markov chain. We establish the well-posedness of the associated system of Hamilton-Jacobi-Bellman (HJB) equations and provide a characterization of the optimal policy. The policy improvement and the convergence of the policy iterations are rigorously established by analyzing the system of equations. We also show that the value function in the exploratory formulation converges to the one in the classical formulation as the temperature parameter vanishes. Finally, a model-free reinforcement learning algorithm is devised and implemented by invoking the policy evaluation based on the martingale characterization. Our numerical examples with financial applications illustrate the effectiveness and efficiency of the proposed RL algorithm.
 
\end{abstract}

\section{Introduction}
The optimal switching problem across multiple regimes entails solving a stochastic optimization problem in which the admissible strategies are formalized by sequences of discrete interventions. A decision-maker in this context faces two basic questions: (i) when to switch from the current regime to another, and (ii) which regime to select when the decision of switching is made. These problems are characterized by their hybrid nature, combining continuous state dynamics with discrete control actions, where each switch between regimes typically incurs a cost while different regimes yield different reward outcomes.  Over past decades, the optimal switching problem has found extensive applications across different fields. Seminal work includes  \cite{carmona2008pricing} on pricing asset scheduling,  \cite{carmona2010valuation} on energy storage valuation, \cite{porchet2009valuation} on power plant valuation, and \cite{olofsson2022management} on hydropower production planning, among others.

The classical stochastic control approach typically assumes a full knowledge of the underlying model. However, this assumption often turns out to be unrealistic in practical applications. RL offers a powerful framework for learning optimal strategies in the unknown environment through trial-and-error interactions. While most conventional RL algorithms are designed in discrete-time settings, many real-life applications evolve continuously in time, motivating an upsurge of recent progress in theories and algorithms of the continuous-time RL approach. Within the continuous-time framework, decision-makers face the fundamental exploration-exploitation trade-off in a continuous-time manner: whether to exploit current knowledge by executing the best-known policy or to explore alternative actions to gather information for potential long-term improvement.  \cite{wang2020reinforcement} addressed this problem by introducing an entropy-regularization on the randomized policy to encourage the exploration. This fundamental study spurred further pioneer studies of theories and algorithms in the continuous-time exploratory framework including \cite{jia2022policy,jia2022gradient,jia2023q}, laying the foundations for the policy evaluation, the policy gradient, and the continuous-time q-learning, respectively. Later, the well-posedness of the exploratory HJB equation, the convergence of policy iterations and the regret analysis have also been examined in \cite{tang2022exploratory, huang2025convergence, tran2025policy,t-z-regret}.

In addition, vast extensions and applications of continuous-time RL algorithms in different context have been considered in the recent literature.  To name a few, \cite{wu2024reinforcement} addressed the continuous-time mean-variance portfolio selection problem in regime-switching markets with unobservable states using reinforcement learning approach; \cite{bo2025optimal} extended the q-learning theory in the model of reflected diffusion processes and applied it to learn the optimal tracking portfolio in incomplete markets; \cite{wei2025continuous} generalized the continuous-time q-learning to mean-field control problems in McKean-Vlasov diffusion models; \cite{wyy2024} studied the continuous-time q-learning for both mean-field control and mean-field game problems from the perspective of the representative agent; 
\cite{RWYZ26} and \cite{RWYZ26-2} further developed the continuous-time q-learning for mean-field control with controlled common noise;
\cite{gao2024reinforcement} studied the extension of q-learning in jump-diffusion models; \cite{bo2024continuous} examined the same jump-diffusion model by invoking the Tsallis entropy;   \cite{dong2024randomized} investigated the optimal stopping in an exploratory framework by considering the randomization of stopping time via the intensity control; \cite{dianetti2024exploratory} utilized the randomization of stopping times as singular control and studied its exploratory formulation under residual entropy regularization; \cite{dai2024learning} exploited the penalization method to transform the optimal stopping problem to an optimal control problem for which the entropy regularization is formalized; \cite{liang2025reinforcement} proposed a continuous-time RL framework for singular stochastic control problems without entropy regularization, characterizing the optimal control through singular control laws;  \cite{liang2025reinforcement-2} further proposed a type of randomization of the singular control laws in \cite{liang2025reinforcement} by considering an auxiliary singular control and entropy regularization, which lead to a time-inconsistent two-stage optimal control problem such that the task is to learn the time-consistent equilibrium.

Despite these advancements of continuous-time RL in different model setups, its application to optimal regime switching problems remains relatively underexplored. This paper studies the exploratory formulation of the optimal regime switching with multiple regimes and bridges its connection to the classical optimal switching problem as the entropy regularization vanishes. To this end, we propose a type of exploratory formulation where the decision-maker randomizes both switching time and the selection of the targeted regime state by invoking a  generator matrix of an associated continuous-time Markov chain (CTMC)  defined on finite state space. The entropy regularization on the generator is imposed to encourage the exploration.  Specifically, we utilize the inherent property of the CTMC—particularly its jump times and state transitions—to determine the switching decision. This formulation, governed by the control of the CTMC's generator matrix, transformed the randomized switching problem into an optimal control problem.

We summarize the main contributions of the present paper as follows:
\begin{itemize}
\item[(i)] We derive the system of exploratory HJB equations and establish the existence of a bounded classical solution to the entropy regularized system (see Lemma \ref{thm:sol-HJB}) by resorting to some partial differential equation (PDE) theories together with a tailor-made truncation argument. Furthermore, we prove its uniqueness and show by a verification theorem (see Proposition \ref{thm-ver}) that this solution coincides with the value function.

\item[(ii)] For the entropy regularized problem, we obtain the characterization of the optimal policy and the policy iteration improvement result in Proposition \ref{thm:improvement}. By analyzing the system of PDEs, we further establish the convergence of the policy iteration to the optimal policy in Theorem \ref{thm:convergence} for the entropy regularized problem with an explicit convergence rate, which is new to the literature. 

\item[(iii)] It is also shown that the value function in the exploratory formulation converges to the value function of the optimal switching problem in the classical formulation as the temperature parameter tends to zero. To facilitate the proof, we resort to some delicate stability analysis of viscosity solutions of the PDE system; see Lemma \ref{lem:uper-lower} and Theorem \ref{thm:convergence-lambda}. This theoretical convergence justifies the choice of small temperature parameter in practical implementations in order to better approximate the optimal decision of switching in the original problem without entropy.

\item[(iv)] Thanks to the policy iteration convergence, a model-free RL algorithm is devised by invoking a policy evaluation and the associated martingale characterization. We resort to a proper stochastic approximation when using the martingale orthogonality condition. An explicit error analysis of this stochastic approximation method is obtained in Theorem \ref{thm:error}. To illustrate the effectiveness of our proposed RL algorithm, we conducted numerical experiments in two examples with satisfactory iteration convergence, both necessitate the application of neural networks to parameterize the targeted functions.  
\end{itemize}

Most existing studies on optimal switching resort to the viscosity solution approach due to the lack of global regularity of solution to the system of variational inequalities, which often impose milder model assumptions. In contrast, our proposed entropy regularization shift the study from the variational inequality system to the PDE system. By imposing some stronger regularity conditions on model coefficients and reward functions and the uniform ellipticity condition (see Assumptions \ref{assump:Lip}, \ref{assump:bound} and  \ref{assump:Holder}), we are allowed to prove the existence of bounded classical solution to the system of exploratory PDEs. We emphasize that these model assumptions are essential to guarantee the desired regularity of solution to the exploratory PDE system and facilitate our technical proofs in establishing some key theoretical foundations for the model-free RL algorithm, namely the policy improvement property (Proposition \ref{thm:improvement}) and the explicit super-exponential convergence rate of policy iteration (Theorem \ref{thm:convergence}).

Let us also briefly compare the present work with three recent related studies. \cite{denkert2025control} introduced a control randomization method without entropy regularization in continuous-time RL with the application to optimal switching problems. They developed an Actor-Critic policy gradient algorithm that alternately learns the value function and the optimal intensity policy. In contrast, our paper propose a different randomization approach for the optimal switching problem, utilizing the generator matrix of a CTMC and incorporating entropy regularization to encourage the exploration. A key advantage of our formulation is that the optimal policy depends explicitly on the value function itself, without requiring any of its derivatives. This allows us to parameterize both the policy and the value function using the same set of parameters. More recently, \cite{dai2025reinforcement} developed a RL approach to identify arbitrage strategies in stock index futures. Following the randomization method in   \cite{dong2024randomized}, they randomized the switching times in \cite{dai2025reinforcement} using the Cox processes and formulated the problem as an optimal switching problem with three regimes where the state process is independent of the regimes. In comparison, we consider an exploratory framework for a more general multi-regime optimal switching problem, where the state process dynamics can also depend on the regime states. Furthermore,  we rigorously establish the convergence of the policy iterations with an explicit convergence rate and also show the convergence as the entropy regularization vanishes. Finally, our work differs from \cite{cao2025two}, which studied a randomization scheme for impulse control problems characterized by fixed points of compound operators combining regularized nonlocal and stopping operators. In contrast, our distinct exploratory formulation leads to the study of PDE system instead of a single PDE problem, for which we need to develop some delicate analysis for the system of equations to deduce some desired convergence results. 

The remainder of this paper is organized as follows. Section \ref{sec:classical} reviews the classical optimal switching problem and presents preliminary results on viscosity solutions to the associated system of HJB variational inequalities. Section \ref{sec:exploratory} introduces the exploratory formulation of the optimal switching problem, providing a regularity analysis of the value function and the characterization of the optimal policy. Section \ref{sec:convergence} establishes the policy improvement and the convergence result of policy iterations. Moreover, the convergence of the exploratory solution as the temperature parameter vanishes is also discusses therein. Section \ref{sec:algorithm} develops a reinforcement learning algorithm that implements the martingale-based policy evaluation and the previous policy iteration, accompanied by an error analysis for the proposed algorithm. Finally, Section \ref{sec:numerical} presents some numerical examples demonstrating the very satisfactory performance of our RL algorithm.

\ \\
\noindent{\bf Notations.}\quad We specify the following list of notations for the rest of this paper.
\begin{itemize}
\item $\R^n$ denotes the $n$-dimensional Euclidean space. For all $x=(x_1,\cdots,x_n),y=(y_1,\cdots,y_n)\in\R^n$, we denote by $\cdot$ the scalar product and by $|\cdot|$ the Euclidean norm: 
\begin{align*}
x\cdot y=\sum_{i=1}^n x_iy_i,\quad |x|=\sqrt{x\cdot x}=\sqrt{\sum_{i=1}^nx_i^2}.
\end{align*}
\item $\R^{n\times d}$ is the set of real-valued $n\times d$ matrices. For $\sigma\in \R^{n\times d}$, we denote by $\sigma^{\top}$ the transpose matrix of $\sigma$. For $A=(a_{ij})_{1\leq i,j\leq n}\in \R^{n \times n}$, $\text{tr}(A)=\sum_{i=1}^n a_{ii}$ is the trace of $A$. We define the matrix norm on $ \R^{n\times d}$ as $|\sigma|=(\text{tr}(\sigma\sigma^{\top}))^{\frac{1}{2}}$.
\item For ${\cal O}\subset \R^n$, $C^k({\cal O})$ is the space of all real-valued continuous functions on ${\cal O}$ with continuous derivatives up to order $k$. For $T\geq 0$, $C^{1,2}([0,T]\times {\cal O})$ is the space of real-valued functions $u$ on $[0,T]\times {\cal O}$ whose partial derivatives $\frac{\partial u}{\partial t},\frac{\partial u}{\partial x_i},\frac{\partial^2 u}{\partial x_ix_j}$, $1\leq i,j\leq n$, exist and are continuous on $[0,T]\times {\cal O}$. For $u\in C^2({\cal O})$, we denote by $D_x u$ the gradient vector of $u$ and $D_x^2u$ the Hessian matrix of $u$.
\item For points $P=(t,x),P'=(t,x)\in[0,T]\times \R^n$, we define the parabolic distance between $P$ and $P'$ by
\begin{align*}
d(P,P')=(|t-t'|+|x-x'|^2)^{\frac{1}{2}}.
\end{align*}
\item For ${\cal D}\subset [0,T]\times \R^n$ and $\alpha\in(0,1)$ we introduce the following norms for functions defined on ${\cal D}$:
\begin{align*}
&||u||_{C^0({\cal D})}=\sup_{P\in {\cal D}}|f(P)|,\quad ||u||_{C^\alpha({\cal D})}=||u||_{C^0({\cal D})}+\sup_{P,P'\in{\cal D},P\neq P'}\frac{|u(P)-u(P')|}{d(P,P')^{\alpha}},\nonumber\\
&||u||_{C^{1}({\cal D})}=||u||_{C^0({\cal D})}+\sum_{i=1}^n\left|\left|\frac{\partial u}{\partial x_i}\right|\right|_{C^0({\cal D})},\quad ||u||_{C^{1+\alpha}({\cal D})}=||u||_{C^\alpha({\cal D})}+\sum_{i=1}^n\left|\left|\frac{\partial u}{\partial x_i}\right|\right|_{C^\alpha({\cal D})},\\
&||u||_{C^{2}({\cal D})}=||u||_{C^1({\cal D})}+\sum_{i=1}^n\left|\left|\frac{\partial u}{\partial x_i}\right|\right|_{C^1({\cal D})}+\left|\left|\frac{\partial u}{\partial t}\right|\right|_{C^{0}({\cal D})}, \\
&||u||_{C^{2+\alpha}({\cal D})}=||u||_{C^{1+\alpha}({\cal D})}+\sum_{i=1}^n\left|\left|\frac{\partial u}{\partial x_i}\right|\right|_{C^{1+\alpha}({\cal D})}+\left|\left|\frac{\partial u}{\partial t}\right|\right|_{C^{\alpha}({\cal D})}.
\end{align*}
We say that function $u(t,x)$ is in $C^{q}({\cal D})$ if $||u||_{C^{q}({\cal D})}$ is finite ($q=0,\alpha,1+\alpha,2+\alpha$).
\end{itemize}

\section{Classical Optimal Switching Problem}\label{sec:classical}
This section first reviews the classical optimal switching problem and introduce some preliminary results on viscosity solutions to the associated system of HJB variational inequalities.

 We fix a complete probability space $(\Omega,\mathcal{F},\Px)$, supporting a $d$-dimensional standard Brownian motion $W =(W_t)_{t\geq 0}$. We denote by $\mathbb{F}$ the complete and right continuous filtration generated by $W$. The terminal time is denoted by $T>0$. Let us introduce the domain ${\cal D}:=[0,T)\times \R^n$, then the closure of ${\cal D}$ is given by $\overline{{\cal D}}=[0,T]\times \R^n$.

We then define the set $\mathcal{A}_t$ of admissible switching strategies at time $t\in[0,T]$ as the set of double sequences $\alpha = (\tau_k,\kappa_k)_{k\geq 0}$, where $\left(\tau_k\right)_{k \geq 0}$ is a non-decreasing sequence of $\mathbb{F}$-stopping times with $\tau_0=t$ and $\lim _{k \rightarrow \infty} \tau_k>T$; $\kappa_k$ is an $\mathcal{F}_{\tau_k}$-measurable random variable valued in the set $\mathbb{I}_m=\{1,2,\cdots,m\}$. With a strategy $\alpha=\left(\tau_k, \kappa_k\right)_{k \geq 0} \in \mathcal{A}_t$ and an initial regime value $i \in \mathbb{I}_m$, we associate the process $(I_s^{t,i})_{s \geq t}$ defined by
\begin{align}\label{eq:alpha}
I_s^{t,i}=\sum_{k \geq 0} \kappa_k { \bf 1}_{s\in[\tau_k, \tau_{k+1})},~ s \geq t, \quad I^{t,i}_{t-}=\kappa_0=i.
\end{align}
 Given $(t,x,i)\in[0,T]\times \R^n\times \mathbb{I}_m$, and a switching control $\alpha\in\mathcal{A}_t$, we consider the controlled diffusion $X^{t,x,i,\alpha}=(X_s^{t,x,i,\alpha})_{s\in[t,T]}$  governed by the SDE:
\begin{align}\label{eq:X-classical}
&dX^{t,x,i,\alpha}_{s}=\mu(s,X^{t,x,i,\alpha}_{s},I_s^{t,i})ds+\sigma(s,X^{t,x,i,\alpha}_{s},I_s^{t,i})dW_{s},\quad s\in(t,T],
\end{align}
with $X_t^{t,x,i,\alpha}=x$. We have the following assumptions for the model coefficients.
\begin{assumption}\label{assump:Lip}
\begin{itemize}
\item[(i)] The drift $\mu(\cdot,\cdot,\cdot):[0,T]\times\R^n\times  \mathbb{I}_m \to \R^n$ and volatility  $\sigma(\cdot,\cdot,\cdot):[0,T]\times\R^n\times  \mathbb{I}_m \to \R^{n\times d}$ are 
uniformly Lipschitz continuous with respect to $x$, i.e.,  there exists a constant $L>0$ such that
\begin{align}\label{eq:Lip}
|\mu(s,x_1,i)-\mu(s,x_2,i)|+|\sigma(s,x_1,i)-\sigma(s,x_2,i)|\leq L|x_1-x_2|
\end{align}
for all $(s,x_1,x_2,i)\in[0,T]\times\R^{2n}\times  \mathbb{I}_m $. 
\item[(ii)] The drift $\mu$ and volatility  $\sigma$ have
linear growth in $x$, i.e.,  there exists a constant $C>0$ such that
\begin{align*}
|\mu(s,x,i)|+|\sigma(s,x,i)|\leq C(1+|x|)
\end{align*}
for all $(s,x,i)\in[0,T]\times\R^{n}\times  \mathbb{I}_m $.
\item [(iii)] There exist some constant $\sigma_0>0$ such that, for all $(t,x,i)\in \overline{{\cal D}}\times\mathbb{I}_m$ and $ \xi\in \R^n$,
\begin{align*}
 \xi\sigma(t,x,i)\sigma^{\top}(t,x,i)\xi^{\top}\geq \sigma_0 \xi\xi^{\top}.
\end{align*}
\end{itemize}
\end{assumption}

Given the initial state $(t,x,i)$,
the expected total profit under the switching strategy $\alpha=\left(\tau_k, \kappa_k\right)_{k \geq 0} \in \mathcal{A}_t$ is given by
\begin{align}\label{eq:J-C}
J_i(t,x; \alpha) = \mathbb{E}\bigg[ \int_t^T  f(s,X_s^{t,x, i,\alpha},I_s^{t,i})ds - \sum\limits_{k=1}^\infty g_{\kappa_{k-1}\kappa_k}{\bf 1}_{\{\tau_k\leq T\}} +h(X_T^{t,x, i,\alpha})\bigg],
\end{align}
where $f(\cdot,\cdot,\cdot) : [0,T]\times \R^n\times\mathbb{I}_m  \to \R$ is the running reward function, $h(\cdot):\R^n \to \R$  is  the terminal reward function, and each $g_{ij}$ denotes the constant cost of switching from regime $i$ to $j$ for all $i\neq j$. We also impose the following assumptions.
\begin{assumption}\label{assump:bound}
\begin{itemize}
\item[(i)]For $i\in \mathbb{I}_m$, the running reward $f(\cdot,\cdot,i)$ and terminal reward  $h(\cdot)$ are assumed to be continuous.  Furthermore, there exists a constant $K_{f,h}>0$ such that
\begin{align}\label{eq:bound-fh}
|f(t,x,i)|+|h(x)|\leq K_{f,h},\quad \forall (t,x,i)\in[0,T]\times \R^n\times\mathbb{I}_m.
\end{align}
\item[(ii)]  For  $i,j\in\mathbb{I}_m$ with $j\neq i$, the cost for switching from regime $i$ to $j$ is positive, that is, $g_{ij}>0$, with the convention $g_{ii}=0$. For $i,j,k\in\mathbb{I}_m$ with $j\neq i,k$, it is less expensive to switch directly in one step from regime $i$ to $k$ than in two steps via an intermediate regime $j$, that is, $g_{ik}<g_{ij}+g_{jk}$. 
\end{itemize}
\end{assumption}
The objective is to maximize the expected total profit over all switching strategies $\alpha$. 
Accordingly, the classical value functions is defined by
\begin{align}\label{eq:value-func-C}
V_i(t,x) = \sup\limits_{\alpha\in\mathcal{A}_t} J_i(t,x; \alpha), \quad (t,x,i)\in[0,T]\times\R^n\times \mathbb{I}_m.
\end{align}

By dynamic programming argument, the associated system of HJB variational inequalities is written by, for $i\in\mathbb{I}_m$,
\begin{align}\label{eq:HJB-VI}
\begin{cases}
\displaystyle \min\left\{-\frac{\partial V_i(t,x)}{\partial t}-\mathcal{L}^i_x V_i(t,x)-f(t,x,i),V_i(t,x)-\max_{j\neq i}(V_{j}(t,x)-g_{ij}) \right\}=0,\quad (t,x)\in{\cal D},\\
\displaystyle V_i(T,x)=h(x),\quad x\in\R^n,
\end{cases}
\end{align}
where the operator $\mathcal{L}_x^i$ is defined by
\begin{align*}
\mathcal{L}_x^i l(t,x) := \mu (t,x,i)D_x l(t,x)+\frac{1}{2}\text{tr}(\sigma\sigma^{\top} (t,x,i)D_x^2 l(t,x)),\quad \text{for}~l(t,\cdot)\in C^2(\R^n).
\end{align*} 
The value function $(V_1,\cdots,V_m)$ can be characterized as the viscosity solution of system \eqref{eq:HJB-VI}, which is defined as below.
\begin{definition}\label{def:viscosity-VI}
Let $(u_1,\cdots,u_m)$ be a $m$-uplet of functions defined on $\overline{{\cal D}}$, $\R$-valued and such that $u_i(T,x)=h(x)$ for any $(i,x)\in\mathbb{I}_m\times \R^n$. The  $m$-uplet $(u_1,\cdots,u_m)$ is called:
\begin{itemize}
\item[(i)] a viscosity supersolution (respectively, subsolution) of system \eqref{eq:HJB-VI} if, for each $i\in \mathbb{I}_m$, $u_i$ is lower-semicontinuous (respectively, upper-semicontinuous) on ${\cal D}$ and for any $(t_0,x_0)\in{\cal D}$ and any test function $\varphi_i\in C^{1,2}({\cal D})$ such that $(t_0,x_0)$ is a local minimum point of $u_i-\varphi_i$ (respectively, maximum), we have
\begin{align*}
\min\Bigg\{&-\frac{\partial \varphi_i(t_0,x_0)}{\partial t}-\mathcal{L}^i_x \varphi_i(t_0,x_0)-f(t_0,x_0,i),\nonumber\\
&\qquad\qquad\quad u_i(t_0,x_0)-\max_{j\neq i}(u_{j}(t_0,x_0)-g_{ij}) \Bigg\}\geq 0~ \text{(respectively, $\leq 0$)};
\end{align*}
\item[(ii)] a viscosity solution of system \eqref{eq:HJB-VI} if it both a viscosity supersolution and subsolution.
\end{itemize}
\end{definition}

Using a similar proof of Theorem 5.1  in  \cite{el2013stochastic}, we have the comparison principle for the system \eqref{eq:HJB-VI}. 
\begin{lemma}[Comparison Principle]\label{lem:comparison}
Suppose Assumptions \ref{assump:Lip} and \ref{assump:bound} hold. Let $(u_1,\cdots,u_m)$  be a bounded   viscosity supersolution of system \eqref{eq:HJB-VI}  and $(v_1,\cdots,v_m)$ be a bounded   viscosity subsolution of system \eqref{eq:HJB-VI}. Then $v_i(t,x)\leq u_i(t,x)$ for all $(t,x,i)\in \overline{{\cal D}}\times \mathbb{I}_m$.
\end{lemma}

Lemma \ref{lem:comparison} will help the proof of uniqueness of viscosity solution. The next result relates the value function  $(V_1,\cdots,V_m)$  to the system of variational inequalities.

\begin{theorem}\label{thm:viscosity}
Under Assumptions \ref{assump:Lip} and \ref{assump:bound}, the value function $(V_1,\cdots,V_m)$ given by \eqref{eq:value-func-C} is the unique bounded viscosity solution of system \eqref{eq:HJB-VI}.
\end{theorem}
\begin{proof}
We begin by proving that the value function $(V_1,\cdots,V_m)$ defined by \eqref{eq:value-func-C} is bounded.  By Assumption \ref{assump:bound}, for any $(i,t,x)\in\mathbb{I}_m\times \overline{{\cal D}}$ and $\alpha\in {\cal A}_t$,
\begin{align*}
J_i(t,x,\alpha)&\leq \mathbb{E}\bigg[ \int_t^T  f(s,X_s^{t,x, i,\alpha},I_s^{t,i})ds  +h(X_T^{t,x, i,\alpha})\bigg]\nonumber\\
&\leq (T-t) K_{f,h}+ K_{f,h},
\end{align*}
which implies $V_i(t,x)\leq (T-t) K_{f,h}+ K_{f,h}$. For the lower bound, consider the no-switching control $\tau_n=\infty$, $n\geq 1$, i.e., $I_s^{t,i}=i$, $s\geq t$. Applying Assumption \ref{assump:bound} again yields
\begin{align*}
V_i(t,x)&\geq  \mathbb{E}\bigg[ \int_t^T  f(s,X_s^{t,x, i,\alpha},I_s^{t,i})ds +h(X_T^{t,x, i,\alpha})\bigg]\nonumber\\
&\geq  -(T-t) K_{f,h}-K_{f,h}.
\end{align*}
Therefore, the value function is bounded. As it is bounded, it follows from Proposition 4.2 in  \cite{bouchard2009stochastic} and Lemma \ref{lem:comparison} that the value function $(V_1,\cdots,V_m)$  is the unique bounded viscosity solution of system \eqref{eq:HJB-VI}.
\end{proof}

 \section{Exploratory Formulation under Entropy Regularization}\label{sec:exploratory}
In this section, we introduce our exploratory formulation of the optimal switching problem, and study the well-posedness of the associated exploratory HJB system as well as the verification theorem.

For the purpose of exploration, we let the agent randomize both the switching times and the choice of regimes to switch to.  This randomization is achieved by considering the control of the generator matrix of a continuous-time finite-state Markov chain $I=(I_t)_{t\in [0,T]}$ on the state space $\mathbb{I}_m$. That is, the agent chooses a feedback policy $\bm{\pi}(t,x) = (\pi_{ij}(t,x))_{i,j\in\mathbb{I}_m}$, where $\pi_{ij}(t,x)$ is the instantaneous intensity of a transition from regime $i$ to regime $j$ at time $t$ when the current state of the controlled diffusion is $X_t=x$. This formulation transforms the original optimal switching problem into a standard optimal control problem over the generator of the Markov chain, where the control acts continuously in time.

For a fixed initial time $t\in[0,T]$, we define the set $\mathbb{U}_t$ of admissible feedback stochastic policies as follows.
\begin{definition}\label{def:admissible}
For $t\in[0,T]$, a policy $\bm{\pi}=(\pi_{ij}(s,x))_{i,j\in\mathbb{I}_m}$ defined on $[t,T]\times\R^n$ belongs to the admissible set $\mathbb{U}_t$ if it satisfies:
\begin{itemize}
\item[(i)] For every $i,j\in\mathbb{I}_m$, the function $\pi_{ij}(\cdot,\cdot):[t,T]\times\R^n \to \R$ is measurable and bounded; moreover, it is uniformly Lipschitz continuous in $x$.
\item[(ii)] For $i\neq j$, $\pi_{ij}(s,x)\ge 0$ for all $(s,x)\in[t,T]\times\R^n$; for $i=j$, $\pi_{ii}(s,x)\le 0$.
\item[(iii)] For every $i\in\mathbb{I}_m$, the row sums satisfy
\[
\sum_{j=1}^m \pi_{ij}(s,x) = 0, \qquad \forall (s,x)\in[t,T]\times\R^n.
\]
\end{itemize}
\end{definition}

Given $(t,x,i)\in[0,T]\times \R^n\times \mathbb{I}_m$,  we consider the controlled diffusion $X=(X_s)_{s\in[t,T]}$  defined by the following SDE:
\begin{align}\label{eq:X-exploratory}
&dX_s=\mu(s,X_{s},I_s)ds+\sigma(s,X_{s},I_s)dW_{s},\quad s\in(t,T].
\end{align}
with $X_t=x$ and $I_{t}=i$. The regime switching process $I=(I_s)_{s\in [t,T]}$ is a continuous-time finite-state Markov chain with state space $\mathbb{I}_m$, whose generator at time $s$ is given by the feedback matrix $\bm{\pi}(s,X_s) = (\pi_{ij}(s,X_s))_{i,j\in\mathbb{I}_m}$.  For $k\geq 1$, denote by $\tau_k$ the $k$-th jump time of process $I$ with $\tau_0=0$ and $\kappa_k:=I_{\tau_k}$.  

The pair $(X,I)$ then forms a regime-switching diffusion with feedback-dependent jump intensities. To make the construction rigorous, we equip the probability space with an independent Poisson random measure ${\cal N}(ds,dz)$ on $[0,T]\times[0,\infty)$ with intensity $ds\,\nu(dz)$, where $\nu$ is a Lebesgue measure on $[0,\infty)$. This random measure allows us to represent the Markov chain $I$ through a stochastic integral with respect to ${\cal N}$ (c.f. Section 2.2 in \cite{nguyen2025hybrid}):
\begin{align}\label{eq:I}
dI_s=\int_{0}^{\infty}p(s,X_{s-},I_{s-},z){\cal N}(ds,dz),\quad s\in[t,T],
\end{align}
where the function $p: [0,T]\times\mathbb{R}^n \times \mathcal{I}_m \times [0,\infty) \mapsto \mathbb{R}$ is given by
\begin{align}\label{eq:p}
p(t,x, i, z)=\sum_{j=1}^{m}(j-i) 1_{\left\{z \in \Delta_{i j}(t,x)\right\}} .
\end{align}
Here, for $j \neq i$, the intervals $\Delta_{i j}(t,x)$ are disjoint, left-closed, and right-open subsets of the positive real half-line, each having length $\pi_{i j}(t,x)$. We set $\Delta_{i j}(t,x)=\emptyset$ if $\pi_{i j}(t,x)=0$ for $i \neq j$. Thus, $p(t,x, i, z)=j-i$ if $z \in \Delta_{i j}(t,x)$, and $p(t,x, i, z)=0$ otherwise. Then, under Assumption 2.1 and the conditions in Definition \ref{def:admissible}, the system \eqref{eq:X-exploratory} admits a unique strong solution $(X,I)$. For detailed proofs of such results in the context of hybrid switching diffusions, we refer the reader to Theorem 2.6 and Corollary 2.8 in Section 2.1, as well as to the discussion of the nonhomogeneous case in Section 2.8 of \cite{nguyen2025hybrid}.

For ${\bm \pi}\in\mathbb{U}_t$, denote by $\bm{\pi}_s=(\pi_{ij}(s,X_s))_{i,j\in\mathbb{I}_m}$. 
To encourage the exploration, we adopt the normalized entropy similar to \cite{dong2024randomized} that $R({\bm \pi},i):=\sum_{j\neq i} \pi^{ij}-\pi^{ij}\log \pi^{ij}$ for $i\in\mathbb{I}_m$. The exploratory formulation of objective functional under entropy regularizer is given by, for $(t,x,i)\in[0,T]\times \R^n\times \mathbb{I}_m$ and ${\bm \pi}\in\mathbb{U}_t$,
\begin{align}\label{eq:J-exploratory}
	&J_{i}^{\lambda}(t,x;{\bm \pi}):=\mathbb{E}_{t,x,i}\bigg[\int^{T}_{t}f(s,X_{s},I_s)ds- \sum_{k=1}^\infty g_{\kappa_{k-1}\kappa_k}{\bf 1}_{\{\tau_k\leq T\}}+\lambda \int^{T}_{t}R({\bm \pi}_s,I_s)ds+h(X_{T})\bigg],
\end{align}
where $\Ex_{t,x,i}[\cdot]:=\Ex[\cdot|X_t=x,I_{t}=i]$, and $\lambda>0$ is the temperature parameter. Furthermore, the optimal value function is denoted by
\begin{align}\label{eq:V}
	V_i^{\lambda}(t,x)=\sup_{{\bm \pi}\in \mathbb{U}_t}J_{i}^{\lambda}(t,x;{\bm \pi}).
\end{align}
%
%
%
%
  Applying the dynamic programming arguments (c.f. Section 5.3.2 in \cite{pham2009continuous}), we derive the system of coupled HJB equations as follows: for $i\in\mathbb{I}_m$,
\begin{align}\label{eq:HJB-V}
\begin{cases}
 \displaystyle \frac{\partial V^{\lambda}_i(t,x)}{\partial t}+\mathcal{L}^i_x V^{\lambda}_i(t,x)+f(t,x,i)\\
\displaystyle\qquad\quad+\sup_{{\bm \pi}_i}\left\{\sum_{j\neq i}\pi_{ij}(V^{\lambda}_{j}(t,x)-g_{ij}-V^{\lambda}_i(t,x))+\lambda\sum_{j\neq i}(\pi_{ij}-\pi_{ij} \log \pi_{ij})\right\}=0,~(t,x)\in {\cal D},\\
 \displaystyle V^{\lambda}_i(T,x)=h(x),~x\in\R^n.
\end{cases}
\end{align}

 Using the first-order condition, we arrive at the characterization of the optimal feedback policy by
\begin{align}\label{eq:optimal-policy}
    \pi_{ij}^*(t,x)=\exp\left(\frac{V^{\lambda}_{j}(t,x)-g_{ij}-V^{\lambda}_{i}(t,x)}{\lambda}\right),\quad j\in\mathbb{I}_m \setminus \{i\},~(t,x)\in \overline{{\cal D}}.
\end{align}
Plugging \eqref{eq:optimal-policy} into \eqref{eq:HJB-V}, we get
\begin{align}\label{eq:HJB-V-policy}
 \frac{\partial V^{\lambda}_i(t,x)}{\partial t}+\mathcal{L}^i_x V^{\lambda}_i(t,x)+f(t,x,i)+\lambda\sum_{j\neq i}\exp\left(\frac{V^{\lambda}_{j}(t,x)-g_{ij}-V^{\lambda}_{i}(t,x)}{\lambda}\right)=0,~(t,x)\in {\cal D},
\end{align}
with the terminal condition $V^{\lambda}_i(T,x)=h(x)$ for $x\in\R^n$.

To establish the existence of classical solution to the HJB system \eqref{eq:HJB-V}, we impose the following assumption.
\begin{assumption}\label{assump:Holder}
The running reward function $f(\cdot,\cdot,i)\in C^{\alpha}({\cal D})$ for $i\in \mathbb{I}_m$ and terminal reward function $h(\cdot)\in C^{2+\alpha}(\R^n)$.
\end{assumption}
\begin{lemma}\label{thm:sol-HJB}
Let Assumptions \ref{assump:Lip}, \ref{assump:bound} and \ref{assump:Holder} hold. Then for any $\lambda>0$, the system of HJB equations \eqref {eq:HJB-V} has a classical solution $(V^{\lambda}_1,V^{\lambda}_2,\cdots,V^{\lambda}_m)$ with $V^{\lambda}_i\in C^{1,2}({\cal D}) \cap C(\overline{{\cal D}})$ for $i\in\mathbb{I}_m$.
\end{lemma}
\begin{proof}
 Given $M > 0$, consider a smooth and non-decreasing cut-off function $\phi_M$ such that $\phi_M (x) = e^x$ for $x \leq  M$, $\phi_M(x)\leq e^{M+1}$ for $x\in(M,M+1)$ and $\phi_M (x) = e^{M+1}$ for $x \geq  M + 1$. Hence, $\phi_M$ is bounded and Lipschitz continuous. Define ${\cal D}_N := \{(t,x):(t,x)\in{\cal D},|x|< N\}$, and let $\partial \mathcal{D}_N$  denote the parabolic boundary of $\mathcal{D}_N$, i.e., the set $([0,T] \times \{ |x| = N\}) \cup (\{T\} \times \{ |x| \le N\})$. First, we will solve the following PDE system: for $i \in\mathbb{I}_m$,
\begin{align}\label{eq:HJB-N-M}
\begin{cases}
\displaystyle \frac{\partial V_i^{M,N}(t,x)}{\partial t}+\mathcal{L}^i_x V_i^{M,N}(t,x)+f(t,x,i)+\lambda \sum_{j\neq i}\phi_M\left(\frac{V^{M,N}_{j}(t,x)-V^{M,N}_{i}(t,x)-g_{ij}}{\lambda}\right)=0,\\[1em]
\displaystyle \hfill (t,x)\in{\cal D}_N,\\[0.5em]
\displaystyle V_i^{M,N}(t,x)=K (T-t)+h(x),\quad (t,x)\in\partial {\cal D}_N,
\end{cases}
\end{align}
where the constant $K>0$ is given by
\begin{align}\label{eq:K}
K:=K_{f,h}+\lambda\sup_{i\in\mathbb{I}_m} \left(\sum_{j\neq i} \exp\left(-\frac{g_{ij}}{\lambda}\right)\right).
\end{align}
For $i\in \mathbb{I}_m$, let us introduce the function
\begin{align*}
u_i(t,x)&=K (T-t)+K_{f,h},\quad (t,x)\in\overline{{\cal D}}_N .
\end{align*}
It follows from Assumption \ref{assump:bound} that
\begin{align}
&\frac{\partial u_i(t,x)}{\partial t}+\mathcal{L}^i_x u_i(t,x)+f(t,x,i)+\lambda \sum_{j\neq i}\phi_M\left(\frac{u_{j}(t,x)-u_{i}(t,x)-g_{ij}}{\lambda}\right)\nonumber\\
&=-K+f(t,x,i)+\lambda \sum_{j\neq i}\phi_M\left(-\frac{g_{ij}}{\lambda}\right)\leq 0,\quad \forall (t,x)\in{\cal D}_N,
\end{align}
and $u_i(t,x)\geq V^{\lambda}_i(t,x)$ for all $(t,x)\in\partial {\cal D}_N$. Similarly, we have 
\begin{align}
&\frac{\partial (-u_i(t,x))}{\partial t}+\mathcal{L}^i_x (-u_i(t,x))+f(t,x,i)+\lambda \sum_{j\neq i}\phi_M\left(\frac{(-u_{j}(t,x))-(-u_{i}(t,x))-g_{ij}}{\lambda}\right)\nonumber\\
&=K+f(t,x,i)+\lambda \sum_{j\neq i}\phi_M\left(-\frac{g_{ij}}{\lambda}\right)\geq 0,\quad \forall (t,x)\in{\cal D}_N,
\end{align}
and $-u_i(t,x)\leq V_i^{\lambda}(t,x)$ for all $(t,x)\in\partial {\cal D}_N$.  Invoking Theorem 2.1 in \cite{kusano1965first}, we obtain that system \eqref{eq:HJB-N-M} has a classical solution $(V_1^{M,N},\cdots,V_m^{M,N})$, with  $V_i^{M,N}\in C^{1+\delta}(\overline{{\cal D}}_N)$ for any $\delta\in(0,1)$ and $V_i^{M,N}\in C^{2+\alpha}(\overline{{\cal D}}_N)$. Furthermore,  we deduce from the comparison theorem (Theorem 1.3 in  \cite{kusano1965first}) that
\begin{align}\label{eq:bound-V-MN}
 |V_i^{M,N}(t,x)|\leq u_i(t,x)=K(T-t)+K_{f,h},\quad \forall (i,t,x)\in \mathbb{I}_m\times \overline{{\cal D}}_N,
\end{align}
which implies that $V_i^{M,N}(t,x)$ is bounded. Thus, by choosing some $M$ large enough, for each $i\in\mathbb{I}_m$, $V^{N}_i:=V^{M,N}_i$ solves the following PDE
\begin{align}\label{eq:HJB-N}
\begin{cases}
\displaystyle \frac{\partial V_i^{N}(t,x)}{\partial t}+\mathcal{L}^i_x V_i^{N}(t,x)+f(t,x,i)+\lambda \sum_{j\neq i}\exp\left(\frac{V^{N}_{j}(t,x)-V^{N}_{i}(t,x)-g_{ij}}{\lambda}\right)=0,\quad (t,x)\in{\cal D}_N,\\[1em]
\displaystyle V_i^{N}(t,x)=K (T-t)+h(x),\quad (t,x)\in\partial {\cal D}_N.
\end{cases}
\end{align}
First, we apply Lemma 2 in \cite{kusano1965first} to the problem \eqref{eq:HJB-N} to derive for any $\delta\in(0,1)$, 
\begin{align*}
||V_i^N||_{C^{1+\delta}({\cal D}_N)}\leq C\left(1+||f(\cdot,\cdot,i)||_{C^{0}({\cal D}_N)}+||h||_{C^{2}({\cal D}_N)}\right).
\end{align*}
In particular, $||V_i^N||_{C^{\alpha}({\cal D}_N)}$ are bounded independently of $N$. We then apply Lemma 1 in \cite{kusano1965first} to the problem \eqref{eq:HJB-N}, obtaining 
\begin{align*}
||V_i^N||_{C^{2+\alpha}({\cal D}_N)}&\leq C\left(1+||f(\cdot,\cdot,i)||_{C^{\alpha}({\cal D}_N)}+||h||_{C^{2+\alpha}({\cal D}_N)}\right)\nonumber\\
&\leq C\left(1+||f(\cdot,\cdot,i)||_{C^{\alpha}({\cal D})}+||h||_{C^{2+\alpha}({\cal D})}\right).
\end{align*}
In view of the uniform $C^{2+\alpha}(\overline{\mathcal{D}}_N)$ estimates, we deduce that for each fixed $N_0$, the sequence $\{V_i^N\}$ is relatively compact in $C^2(\overline{\mathcal{D}}_{N_0})$. By a diagonal argument, we can extract a subsequence converging uniformly in $\overline{{\cal D}}_{N_0}$ together with the first $x$ , $t$-derivatives  and second $x$-derivatives to a limit function $V_i^{\lambda}$, for every $N_0$. The limit function $V_i^\lambda$ thus belongs to $C^{1,2}(\mathcal{D}) \cap C(\overline{\mathcal{D}})$ and satisfies the PDE on $\mathcal{D}$. The terminal condition is inherited from the boundary condition on $\partial \mathcal{D}_N$ and the continuity at $t=T$. We then complete the proof. 
\end{proof}

By the proof of Lemma \ref{thm:sol-HJB}, for any $\lambda>0$, the system of HJB equations \eqref{eq:HJB-V} admits a classical solution $(V^{\lambda}_1,V^{\lambda}_2,\cdots,V^{\lambda}_m)$ satisfying 
\begin{align}\label{eq:bound-V-lambda}
 |V_i^{\lambda}(t,x)|\leq K(T-t)+K_{f,h},\quad \forall (i,t,x)\in \mathbb{I}_m\times \overline{{\cal D}},
\end{align}
where the constant $K>0$ is given by \eqref{eq:K}. We now prove that this bounded classical solution is unique and coincides with the value function.
\begin{proposition}[Verification Theorem]\label{thm-ver}
Suppose Assumptions \ref{assump:Lip}, \ref{assump:bound}, and \ref{assump:Holder} hold, and let $(V_1^{\lambda},V_2^{\lambda},\cdots,V_m^{\lambda})$ be a bounded classical solution to system \eqref{eq:HJB-V}, as provided by Lemma \ref{thm:sol-HJB}. We define a set of feedback functions by
\begin{align}\label{eq:optimal-policy-j}
    \pi_{ij}^*(t,x)=\exp\left(\frac{V^{\lambda}_{j}(t,x)-g_{ij}-V^{\lambda}_{i}(t,x)}{\lambda}\right),\quad j\in\mathbb{I}_m \setminus \{i\},~(t,x)\in \overline{{\cal D}},
\end{align}
and 
\begin{align}\label{eq:optimal-policy-i}
    \pi_{ii}^*(t,x)=-\sum_{j\neq i}\pi_{ij}^*(t,x),\quad (t,x)\in \overline{{\cal D}}.
\end{align}
Consider the process $X^*$ governed by the dynamics \eqref{eq:X-exploratory}, where the generator of the process $I^*$ is given by $ \pi_{ij}^*(t,X_t^*)_{i,j\in\mathbb{I}_m, t\in[0,T]}$. Then, for each $i\in\mathbb{I}_m$, the function $V_i^{\lambda}$ is the value function for problem \eqref{eq:V}, and for any $t\in[0,T]$, the policy ${\bm \pi}^*\in\mathbb{U}_t$  is optimal.
\end{proposition}
\begin{proof}
For $(i,t,x)\in \mathbb{I}_m\times \overline{\mathcal D}$, $\bm{\pi}\in\mathbb{U}_t$, and $s\in[t,T]$, applying It\^o's rule yields
\begin{align}\label{eq:Ito-V}
V_{I_s}^{\lambda}(s,X_s)&=V_{i}^{\lambda}(t,x)+\int_t^s \left(\frac{\partial V_{I_\ell}^{\lambda}(\ell,X_{\ell})}{\partial t}+\mathcal{L}^{I_{\ell}}_x V_{I_\ell}^{\lambda}(\ell,X_{\ell})\right)d\ell+\int_t^s (D_x V_{I_\ell}^{\lambda} (\ell,X_\ell))^{\top}\sigma(\ell,X_{\ell},I_\ell)dW_{\ell}\nonumber\\
&\quad+ \int_t^s \sum_{j\neq I_\ell}\left(\pi_{I_{\ell}j}(\ell,X_{\ell})(V^{\lambda}_{j}(\ell,X_{\ell})-V^{\lambda}_{I_\ell}(\ell,X_{\ell}))\right)d\ell\nonumber\\
&\quad+
\int_t^s \int_{0}^{\infty}\left(V^{\lambda}_{I_{\ell-}+p(\ell,X_{\ell-}, I_{\ell-},z) }(\ell,X_{\ell-})-V^{\lambda}_{I_{\ell}-}(\ell,X_{\ell-})\right) \tilde{{\cal N}}(d \ell, d z),
\end{align}
where $\tilde{{\cal N}}(ds,dz):={\cal N}(ds,dz)-ds\nu(dz)$ is a martingale measure. By the boundedness of the value function in \eqref{eq:bound-V-lambda}, the stochastic integrals with respect to $W$ and $\tilde{{\cal N}}$ are both martingales.

It follows from \eqref{eq:HJB-V} that for any $(i,t,x)\in\mathbb{I}_m\times[0,T]\times\R^n$,
\begin{align}\label{eq:HJB-ineq}
& \frac{\partial V^{\lambda}_i(t,x)}{\partial t}+\mathcal{L}^i_x V^{\lambda}_i(t,x)+\sum_{j\neq i}\pi_{ij}(V^{\lambda}_{j}(t,x)-g_{ij}-V^{\lambda}_i(t,x))\nonumber\\
&\qquad\geq \sum_{j\neq i}\pi_{ij}g_{ij}-f(t,x,i)-\lambda\sum_{j\neq i}(\pi_{ij}-\pi_{ij} \log \pi_{ij}).
\end{align}
Taking expectations on both sides of \eqref{eq:Ito-V} and applying \eqref{eq:HJB-ineq}, we obtain
\begin{align}\label{eq:value-eq}
V_i^{\lambda}(t,x)\geq \mathbb{E}_{t,x,i}\bigg[\int^{s}_{t}f(\ell,X_{\ell},I_\ell)d\ell- \sum_{k=1}^\infty g_{\kappa_{k-1}\kappa_k}{\bf 1}_{\{\tau_k\leq s\}}+\lambda \int^{s}_{t}R({\bm \pi}_\ell,I_\ell)d\ell+V_{I_s}^{\lambda}(s,X_s)\bigg].
\end{align}
Letting $s\to T$ in \eqref{eq:value-eq} and using the terminal condition $V_i^{\lambda}(T,x)=h(x)$ together with the dominated convergence theorem, we deduce
\begin{align}\label{eq:value-eq-2}
V_i^{\lambda}(t,x)\geq \mathbb{E}_{t,x,i}\bigg[\int^{T}_{t}f(\ell,X_{\ell},I_\ell)d\ell- \sum_{k=1}^\infty g_{\kappa_{k-1}\kappa_k}{\bf 1}_{\{\tau_k\leq T\}}+\lambda \int^{T}_{t}R({\bm \pi}_\ell,I_\ell)d\ell+h(X_T)\bigg].
\end{align}
The inequality \eqref{eq:value-eq-2} holds for any $\bm{\pi}\in\mathbb{U}_t$ and becomes an equality when $\bm{\pi}=\bm{\pi}^*$. Moreover, using the boundedness of the value function in \eqref{eq:bound-V-lambda}, we can verify that $\bm{\pi}^*\in \mathbb{U}_t$ for any $t\in[0,T]$. This completes the proof.
\end{proof}

\section{Policy Iteration and Convergence}\label{sec:convergence}
The goal of this section is to study the policy iteration using the characterization in \eqref{eq:optimal-policy}. In particular, in the context of optimal regime switching, we aim to show the policy improvement and the convergence of policy iterations, which demonstrate that each policy update guarantees the performance enhancement and the repeated iterations will lead to the desired optimal policy when the model is known. We also examine the connection between our exploratory formulation and the classical optimal switching problem by analyzing the limit of the vanishing regularization. 

We first focus on the rule of policy iteration. Given a feedback strategy ${\bm \pi}^n(t,x)=(\pi_{ij}^n(t,x))_{i,j\in \mathbb{I}_m}$, the corresponding value function $(V^{n}_1,\cdots,V^{n}_m)$ satisfies the following PDE system: for $i\in \mathbb{I}_m$,
\begin{align}\label{eq:v-n}
\begin{cases}
 \displaystyle  \frac{\partial V_i^{n}(t,x)}{\partial t}+\mathcal{L}^i_x V_i^{n}(t,x)+f(t,x,i)+H_i({\bm \pi}^n_i(t,x),V_1^n(t,x),\cdots,V_m^n(t,x))=0,\\
 \displaystyle V_i^{n}(T,x)=h(x).
\end{cases}
\end{align}
 Here, ${\bm \pi}_i^n(t,x)=(\pi_{ij}(t,x))_{j\in\mathbb{I}_m}$ for $i\in\mathbb{I}_m$, and the Hamiltomian $H_i({\bm \pi}_i,{\bm y}):\R^{ m}\times\R^m\to \R$ is defined by
\begin{align}\label{eq:H}
H_i({\bm \pi}_i,{\bm y})=\sum_{j\neq i}\pi_{ij}(y_{j}-g_{ij}-y_i)+\lambda\sum_{j\neq i}(\pi_{ij}-\pi_{ij} \log \pi_{ij}).
\end{align}
Having the value function pair $(V^n_1,\cdots,V_m^n)$, one can construct a feedback strategy ${\bm \pi}^{n+1}$ satisfying
\begin{align}\label{eq:policy-n}
    \pi_{ij}^{n+1}(t,x)=\exp\left(\frac{V^n_{j}(t,x)-g_{ij}-V^n_{i}(t,x)}{\lambda}\right),  ~i,j\in\mathbb{I}_m,j\neq i.
\end{align}
We continue this iteration, generating a sequence of strategy-value function pairs. The following result states that each iteration improves the value function.

\begin{proposition}\label{thm:improvement}
Let Assumptions \ref{assump:Lip}, \ref{assump:bound} and \ref{assump:Holder} hold.  Give any initial guess $(V^0_1, \cdots, V^0_m)$ with $V^0_i \in C^0(\overline{{\cal D}})$ for $i\in\mathbb{I}_m$. $\{(V^n_i,\pi^n_{ij})_{i,j\in \mathbb{I}_m}\}_{n=1,2, \ldots}$ are defined iteratively according to \eqref{eq:v-n} and \eqref{eq:policy-n}.  Then, we have that $ V^{n}_i \leq V^{n+1}_i\leq V_i^{\lambda}$ for $i\in\mathbb{I}_m$ and  $n=1,2, \ldots$.
\end{proposition}

\begin{proof}
By Proposition \ref{thm-ver}, the tuple $(V^\lambda_1,\cdots,V_m^{\lambda})$ constitutes the optimal value function for the control problem \eqref{eq:J-exploratory}. Since $(V^{n}_1,\cdots,V^{n}_m)$ is the value function corresponding to the strategy ${\bm \pi}^n(t,x)$ for the same problem, it follows that $V^{n}_i \leq V_i^{\lambda}$ for all $i\in\mathbb{I}_m$ and $n\geq 1$.

Given the uniform bound $V^{n}_i \leq V_i^{\lambda}$, the estimate \eqref{eq:bound-V-lambda}, $V^0_i \in C^0(\overline{{\cal D}})$, and the policy iteration definition \eqref{eq:policy-n}, we deduce that for any $n\geq 1$ and $i,j\in\mathbb{I}_m$, the function $\pi_{ij}^{n}(t,x)$ is bounded on $\overline{{\cal D}}$. For $n\geq 1$, proceeding with a truncation argument analogous to the proof of Lemma \ref{assump:Holder}, we select a suitable constant $\tilde{K}$ and consider the following pair of PDE systems defined on the truncated domain ${\cal D}_N$:
\begin{align}\label{eq:v-n-N}
\begin{cases}
\displaystyle \frac{\partial V_i^{n,N}(t,x)}{\partial t}+\mathcal{L}^i_x V_i^{n,N}(t,x)+f(t,x,i)\\
\displaystyle \quad +\,H_i({\bm \pi}^{n}_i(t,x),V_1^{n,N}(t,x),\cdots,V_m^{n,N}(t,x))=0, & (t,x)\in{\cal D}_N,\\[1em]
\displaystyle V_i^{n,N}(t,x)=\tilde{K} (T-t)+h(x), & (t,x)\in\partial {\cal D}_N,
\end{cases}
\end{align}
and
\begin{align}\label{eq:HJB-n+1-N}
\begin{cases}
\displaystyle \frac{\partial V_i^{n+1,N}(t,x)}{\partial t}+\mathcal{L}^i_x V_i^{n+1,N}(t,x)+f(t,x,i)\\
\displaystyle \quad +\,H_i({\bm \pi}^{n+1,N}_i(t,x),V_1^{n+1,N}(t,x),\cdots,V_m^{n+1,N}(t,x))=0, & (t,x)\in{\cal D}_N,\\[1em]
\displaystyle V_i^{n+1,N}(t,x)=\tilde{K} (T-t)+h(x), & (t,x)\in\partial {\cal D}_N,
\end{cases}
\end{align}
where  policy ${\bm \pi}^{n+1,N}$ is given by
\begin{align}\label{eq:policy-n+1}
    \pi_{ij}^{n+1,N}(t,x)=\exp\left(\frac{V^{n,N}_{j}(t,x)-g_{ij}-V^{n,N}_{i}(t,x)}{\lambda}\right),  ~i,j\in\mathbb{I}_m,~j\neq i.
\end{align}
As in Lemma \ref{assump:Holder}, the bounded solutions satisfy $V^{n,N}_{i} \to V^{n}_{i}$ and $V^{n+1,N}_{i} \to V^{n+1}_{i}$ as $N \to \infty$ for each $i\in\mathbb{I}_m$.

Define the difference $\Delta_i^{n,N}(t,x) := V_i^{n+1,N}(t,x) - V_i^{n,N}(t,x)$ for $i\in\mathbb{I}_m$ and $(t,x)\in\overline{{\cal D}}_N$. Subtracting \eqref{eq:v-n-N} from \eqref{eq:HJB-n+1-N} yields the equation satisfied by $\Delta_i^{n,N}$:
\begin{align}\label{eq:Delta-n}
\frac{\partial \Delta_i^{n,N}(t,x)}{\partial t}+\mathcal{L}^i_x \Delta_i^{n,N}(t,x) 
&+ H_i({\bm \pi}^{n+1,N}_i(t,x),V_1^{n+1,N}(t,x),\cdots,V_m^{n+1,N}(t,x)) \nonumber \\
&- H_i({\bm \pi}^n_i(t,x),V_1^{n,N}(t,x),\cdots,V_m^{n,N}(t,x)) = 0, \quad (t,x)\in {\cal D}_N,
\end{align}
with boundary condition $\Delta_i^{n,N}(t,x)=0$ for $(t,x)\in\partial {\cal D}_N$. A key observation from \eqref{eq:policy-n+1} is that ${\bm \pi}_i^{n+1,N}(t,x)$ maximizes the Hamiltonian:
\begin{align}\label{eq:policy-n-max}
 {\bm \pi}_i^{n+1,N}(t,x) = \argmax_{{\bm \pi} } H_i({\bm \pi}, V_1^{n,N}(t,x),\cdots,V_m^{n,N}(t,x)).
 \end{align}
 Using \eqref{eq:Delta-n} and the maximizing property \eqref{eq:policy-n-max}, we derive for $(t,x)\in {\cal D}_N$:
\begin{align}\label{eq:Delta-n-1}
&\frac{\partial \Delta_i^{n,N}(t,x)}{\partial t}+\mathcal{L}^i_x \Delta_i^{n,N}(t,x)+\sum_{j\neq i}\pi_{ij}^{n+1,N}(t,x)\Delta_j^{n,N}(t,x)-\sum_{j\neq i}\pi_{ij}^{n+1,N}(t,x)\Delta_i^{n,N}(t,x)\nonumber\\
&=-H_i({\bm \pi}_i^{n+1,N}(t,x),V_1^{n+1,N}(t,x),\cdots,V_m^{n+1,N}(t,x))+\sum_{j\neq i}\pi_{ij}^{n+1,N}(t,x)(\Delta_j^{n,N}(t,x)-\Delta_i^{n,N}(t,x))\nonumber\\
&\quad+H_i({\bm \pi}^n_i(t,x),V_1^{n,N}(t,x),\cdots,V_m^{n,N}(t,x))\nonumber\\
&=H_i({\bm \pi}^{n}_i(t,x),V_1^{n,N}(t,x),\cdots,V_m^{n,N}(t,x))-H_i({\bm \pi}^{n+1,N}_i(t,x),V_1^{n,N}(t,x),\cdots,V_m^{n,N}(t,x))\nonumber\\
&\leq 0.
\end{align}
Applying Theorem 1.3 in \cite{kusano1965first}, we deduce that $\Delta^{n,N}_i(t,x)\geq 0$, that is, $V_i^{n+1,N}(t,x)\geq V_i^{n,N}(t,x)$, for all  $i\in\mathbb{I}_m$ and $(t,x)\in\overline{{\cal D}}$. 
Passing to the limit as $N \to \infty$, we have $V_i^{n+1}(t,x)\geq V_i^{n}(t,x)$, which then completes the proof.
\end{proof}

The following theorem, as the first main result of this paper, establishes a fundamental convergence guarantee for our policy iteration method, demonstrating that the sequence of value functions $(V_1^n,\cdots,V_m^n)$ generated through successive iterations converges uniformly to the optimal value functions  $(V_1^{\lambda}\cdots,V^{\lambda}_m)$ of our exploratory optimal switching problem. Moreover, we can obtain the explicit convergence rate for the policy iteration. 

\begin{theorem}\label{thm:convergence}
Let Assumptions \ref{assump:Lip}, \ref{assump:bound} and \ref{assump:Holder} hold.  Give any initial guess $(V^0_1, \cdots, V^0_m)$ with $V^0_i \in C^0(\overline{{\cal D}})$ for $i\in\mathbb{I}_m$. $\{(V^n_i,\pi^n_{ij})_{i,j\in \mathbb{I}_m}\}_{n=1,2, \ldots}$ are defined iteratively according to \eqref{eq:v-n} and \eqref{eq:policy-n}. Then, we have that, for all $n\geq 1$,
\begin{align}\label{eq:convergence-rate}
\sup_{i\in\mathbb{I}_m}\sup_{(t,x)\in\overline{{\cal D}}}|V_i^n(t,x)-V_i^{\lambda}(t,x)|\leq C_1\frac{C_2^n}{n!},
\end{align}
where $C_1,C_2>0$ are constants independent of $n$.
\end{theorem}
\begin{proof}
For $n\geq 0$, let us introduce the function $F^n:[0,T]\to \R_+$ given by
\begin{align}\label{eq:F}
F^n(t):=\sup_{i\in\mathbb{I}_m}\sup_{x\in\R^n}|V_i^n(t,x)-V_i^{\lambda}(t,x)|.
\end{align}
By \eqref{eq:bound-V-MN} and the convergence of $\{V_i^N\}$ to $V_i^\lambda$ as $N\to\infty$, we can obtain 
\begin{align}\label{eq:bound-V}
 |V_i^{\lambda}(t,x)|\leq K(T-t)+K_{f,h}\leq KT+K_{f,h},\quad \forall (i,t,x)\in \mathbb{I}_m\times \overline{{\cal D}},
\end{align}
where the constant $K$ is given by \eqref{eq:K}. This leads to the boundedness of $V_i^{\lambda}(t,x)$, which in turn implies that the policy ${\bm \pi}^*$ from \eqref{eq:optimal-policy} is bounded.  Similarly, by using Proposition \ref{thm:improvement} and \eqref{eq:policy-n}, we can deduce that the sequence of functions $V^n_i(t,x)$ and the corresponding policies ${\bm \pi}^n(t,x)$ are uniformly bounded for $n\geq 1$. Then, it follows from \eqref{eq:optimal-policy}, \eqref{eq:H} and \eqref{eq:policy-n} that
\begin{align}\label{eq:C-star}
&\left|H_i({\bm \pi}_i^{n},V_1^{\lambda},\cdots,V_m^{\lambda})-H_i({\bm \pi}_i^*,V_1^{\lambda},\cdots,V_m^{\lambda})\right|\nonumber\\
&=\Bigg|\sum_{j\neq i}\pi_{ij}^n(V_{j}^{\lambda}-g_{ij}-V_i^{\lambda})+\lambda\sum_{j\neq i}(\pi_{ij}^n-\pi_{ij}^n \log \pi_{ij}^n)\nonumber\\
&\quad+\left(\sum_{j\neq i}\pi_{ij}^*(V_{j}^{\lambda}-g_{ij}-V_i^{\lambda})+\lambda\sum_{j\neq i}(\pi_{ij}^*-\pi_{ij}^* \log \pi_{ij}^*)\right)\Bigg|\nonumber\\
&\leq \sum_{j\neq i}\pi^n_{ij}|V_j^{\lambda}-V_j^n|+|V_i^{\lambda}-V_i^n|\sum_{j\neq i}\pi^n_{ij}\nonumber\\
&\quad+\lambda \sum_{j\neq i}\left|\exp\left(\frac{V^n_{j}-g_{ij}-V^n_{i}}{\lambda}\right)-\exp\left(\frac{V_{j}^{\lambda}-g_{ij}-V_{i}^{\lambda}}{\lambda}\right)\right|\nonumber\\
&\leq C^*F^n(t),
\end{align}
where $C^*>0$ is a constant independent of $n$. For $n\geq 0$, we define the function $w_i^n:\overline{{\cal D}}\to \R$ for $i\in\mathbb{I}_m$ as
\begin{align*}
w_i^n(t,x):=V_i^{\lambda}(t,x)-V_i^{n+1}(t,x)-C^*\int_t^T F^n(s)ds,\quad (t,x)\in\overline{{\cal D}}.
\end{align*}
By using \eqref{eq:C-star}, it holds that for any $(t,x)\in{\cal D}$,
\begin{align*}
&\frac{\partial w_i^{n}(t,x)}{\partial t}+\mathcal{L}^i_x w_i^{n}(t,x)+\sum_{j\neq i}\pi_{ij}^{n+1}(t,x)w_j^n(t,x)-\sum_{j\neq i}\pi_{ij}^{n+1}(t,x)w_i^n(t,x)\nonumber\\
&=H_i({\bm \pi}_i^{n+1}(t,x),V_1^{\lambda}(t,x),\cdots,V_m^{\lambda}(t,x))-H_i({\bm \pi}^*_i(t,x),V_1^{\lambda}(t,x),\cdots,V_m^{\lambda}(t,x))+C^*F^n(t)\geq 0,
\end{align*}
and $w_i^n(T,x)= 0$ for $x\in\R^n$. By virtue of Theorem 1.3 in \cite{kusano1965first}, we deduce $w^n_i(t,x)\leq 0$. That is, 
\begin{align}
V_i^{\lambda}(t,x)-V_i^{n+1}(t,x)\leq C^*\int_t^T F^n(s)ds,\quad \forall (i,t,x)\in\mathbb{I}_m\times\overline{{\cal D}}.
\end{align}
This yields the inequality
\begin{align}
F^{n+1}(t)\leq C^*\int_t^T F^n(s)ds,\quad \forall t\in[0,T],
\end{align}
from which we deduce that
\begin{align}
F^n(t)\leq \frac{(C^*)^nT^n}{n!}\sup_{s\in[0,T]}F^1(s),\quad \forall t\in[0,T].
\end{align}
Because $F^1(t)$ is bounded, let $C_2=C^*T$ and $C_1=\sup_{t\in[0,T]}F^1(t)$. Then we obtain that desired result.
\end{proof}

\begin{remark}
By Stirling's approximation, the convergence rate in \eqref{eq:convergence-rate} has the asymptotic behavior
\begin{align*}
O\!\left(\frac{C_2^{\,n}}{n!}\right) \sim O\!\left(\frac{C_2^{\,n}}{\sqrt{2\pi n}\,(n/e)^n}\right)
= O\!\left(\exp\bigl(n\log C_2 - n\log n + n - \tfrac12\log n\bigr)\right).
\end{align*}
Consequently, it behaves like $\exp(n\log C_2 - n\log n + O(n))$, i.e., the super‑exponential rate that is faster than any exponential rate $e^{-C n}$ with $C>0$.

\end{remark}

To establish a connection between our exploratory formulation and the classical optimal switching problem, we next rigorously analyze the convergence result of the exploratory solution as the temperature parameter $\lambda$ approaches zero. Unlike the existing results in \cite{tang2022exploratory} for regular control problem that focus on a single PDE problem, the nature of problem with multiple regime states calls for some distinct analysis to investigate the system of PDEs in our setting.
In particular, we employ some stability analysis of viscosity solutions to the PDE system to examine the limit of vanishing entropy regularization. The mathematical goal is to show that the solution of the system of PDE will converge to the solution of the system of variational inequalities as $\lambda\rightarrow 0$.


Let us introduce the upper and lower weak limits of functions $(V_1^{\lambda},\cdots,V_m^{\lambda})$ defined as follows: for $i\in\mathbb{I}_m$ and $(t,x)\in\overline{{\cal D}}$,
\begin{align}\label{eq:upper}
\overline{V}_i(t,x):=\begin{cases}
\displaystyle \limsup_{\lambda\to 0\atop (s,y)\to (t,x), (s,y)\in {\cal D}}V_i^{\lambda}(s,y),&(t,x)\in{\cal D},\\[1em]
\displaystyle  \qquad h(x), &t=T,~x\in\R^n,
\end{cases}
\end{align}
and 
\begin{align}\label{eq:lower}
\underline{V}_i(t,x):=\begin{cases}
\displaystyle\liminf_{\lambda\to 0,\atop (s,y)\to (t,x),(s,y)\in {\cal D}}V_i^{\lambda}(s,y),&(t,x)\in{\cal D},\\[1em]
\displaystyle \qquad h(x), &t=T,~x\in\R^n.
\end{cases}
\end{align}

The next lemma plays a crucial role in establishing the convergence of the value functions $(V_1^{\lambda},\cdots,V_m^{\lambda})$ as the temperature parameter $\lambda$ tends to zero. By defining the upper and lower weak limits, we capture the limiting behavior of these functions. The result asserts that these limits are bounded and satisfy the viscosity solution properties for the system of HJB equations \eqref{eq:HJB-VI}. Specifically, the upper weak limits form a viscosity subsolution, and the lower weak limits form a viscosity supersolution.

\begin{lemma}\label{lem:uper-lower}
Let Assumptions \ref{assump:Lip}, \ref{assump:bound}, and \ref{assump:Holder} hold. Consider the upper and lower weak limits of the functions $(V_1^{\lambda},\cdots,V_m^{\lambda})$, defined by \eqref{eq:upper} and \eqref{eq:lower}, respectively. Then the tuple of upper weak limits $(\overline{V}_1,\cdots,\overline{V}_m)$ is a bounded viscosity subsolution of system \eqref{eq:HJB-VI}, while the tuple of lower weak limits $(\underline{V}_1,\cdots,\underline{V}_m)$ is a bounded viscosity supersolution of system \eqref{eq:HJB-VI}.
\end{lemma}
\begin{proof}
It follows from \eqref{eq:K}, \eqref{eq:bound-V-lambda} and Assumption \ref{assump:bound}-(ii) that
\begin{align*}
 |V_i^{\lambda}(t,x)| \leq KT+K_{f,h}
&= K_{f,h}T+\lambda\sup_{i\in\mathbb{I}_m} \left(\sum_{j\neq i} \exp\left(-\frac{g_{ij}}{\lambda}\right)\right)T +K_{f,h}\\
&\leq K_{f,h}T+\lambda(m-1)T +K_{f,h}
\end{align*}
for all $\lambda>0$ and $(i,t,x)\in \mathbb{I}_m\times \overline{{\cal D}}$. 
This implies that $\overline{V}_i$ and $\underline{V}_i$ for $i\in \mathbb{I}_m$ are bounded functions. Applying Lemma 1.5 in Chapter V of \cite{bardi1997optimal}, $\overline{V}_i$ is  upper-semicontinuous on ${\cal D}$ while $\underline{V}_i$ is  lower-semicontinuous on ${\cal D}$ for every $i\in \mathbb{I}_m$.

We next show that the tuple of upper weak limits $(\overline{V}_1,\cdots,\overline{V}_m)$ is a viscosity subsolution of system \eqref{eq:HJB-VI} using the contradiction argument. For $i\in \mathbb{I}_m$, let  $(t_0,x_0)\in{\cal D}$ and the test function $\varphi_i\in C^{1,2}({\cal D})$ such that $(t_0,x_0)$ is a local maximum of $\overline{V}_i-\varphi_i$. Assume that
\begin{align}\label{eq:sub-1}
\min\Bigg\{&-\frac{\partial \varphi_i(t_0,x_0)}{\partial t}-\mathcal{L}^i_x \varphi_i(t_0,x_0)-f(t_0,x_0,i),\nonumber\\
&\qquad\qquad\quad \overline{V}_i(t_0,x_0)-\max_{j\neq i}(\overline{V}_{j}(t_0,x_0)-g_{ij}) \Bigg\}>0.
\end{align}
That is,
\begin{align}
&\delta:=-\frac{\partial \varphi_i(t_0,x_0)}{\partial t}-\mathcal{L}^i_x \varphi_i(t_0,x_0)-f(t_0,x_0,i)>0,\label{eq:delta}\\
&\varepsilon:=\overline{V}_i(t_0,x_0)-\max_{j\neq i}(\overline{V}_{j}(t_0,x_0)-g_{ij})>0.\label{eq:varepsilon}
\end{align}
In view of Lemma 1.6 in Chapter V of \cite{bardi1997optimal}, there exists a sequence $\{(t_n,x_n)\}_{n \geq 1}$ with $(t_n,x_n)\in {\cal D}$ and   a sequence $\{\lambda_n\}_{n \geq 1}$ with $\lambda_n>0$, $\lim_{n\to \infty}\lambda_n=0$ such that $(t_n,x_n)$ is a local maximum point of $V^{\lambda_n}_i-\varphi_i$ and 
\begin{align}\label{eq:upper-lim-point}
\lim_{n\to \infty}(t_n,x_n)=(t_0,x_0),\quad \lim_{n\to \infty}V^{\lambda_n}_i(t_n,x_n)=\overline{V}_i(t_0,x_0).
\end{align}
Lemma \ref{thm:sol-HJB} implies that for any $\lambda>0$, $(V^{\lambda}_1,V^{\lambda}_2,\cdots,V^{\lambda}_m)$ is a classical solution to the system of of HJB equations \eqref {eq:HJB-V}, thus $V^{\lambda_n}_i$ is a viscosity subsolution of the following PDE:
\begin{align*}
-\frac{\partial V^{\lambda_n}_i(t,x)}{\partial t}-\mathcal{L}^i_x V^{\lambda_n}_i(t,x)-f(t,x,i)-\lambda\sum_{j\neq i}\exp\left(\frac{V^{\lambda_n}_{j}(t,x)-g_{ij}-V^{\lambda_n}_{i}(t,x)}{\lambda}\right)=0.
\end{align*}
Consequently, we have 
\begin{align}\label{eq:sub-2}
-\frac{\partial \varphi_i(t_n,x_n)}{\partial t}-\mathcal{L}^i_x \varphi_i(t_n,x_n)-f(t_n,x_n,i)-\lambda_n\sum_{j\neq i}\exp\left(\frac{V^{\lambda_n}_{j}(t_n,x_n)-g_{ij}-V^{\lambda_n}_{i}(t_n,x_n)}{\lambda}\right)\leq0
\end{align}
for any $n\geq 1$.

From \eqref{eq:delta}, \eqref{eq:varepsilon} and \eqref{eq:upper-lim-point}, it follows that there exists some $n_1>0$ such that for all $n\geq n_1$,
\begin{align*}
-\frac{\partial \varphi_i(t_n,x_n)}{\partial t}-\mathcal{L}^i_x \varphi_i(t_n,x_n)-f(t_n,x_n,i)\geq \frac{\delta}{2},
\end{align*}
and for any $j\in\mathbb{I}_m,j\neq i$,
\begin{align*}
V^{\lambda_n}_j(t_n,x_n)-g_{ij}-V^{\lambda_n}_{i}(t_n,x_n) \leq -\frac{\varepsilon}{2}.
\end{align*}
Selecting $n_2$ such that for all $n\geq n_2$, $\lambda_n\exp(-\frac{\varepsilon}{2\lambda_n})<\frac{\delta}{2(m-1)}$, then for $n\geq \max\{n_1,n_2\}$, we get that 
\begin{align}\label{eq:sub-3}
&-\frac{\partial \varphi_i(t_n,x_n)}{\partial t}-\mathcal{L}^i_x \varphi_i(t_n,x_n)-f(t_n,x_n,i)-\lambda_n\sum_{j\neq i}\exp\left(\frac{V^{\lambda_n}_{j}(t_n,x_n)-g_{ij}-V^{\lambda_n}_{i}(t_n,x_n)}{\lambda}\right)\nonumber\\
&\geq -\frac{\partial \varphi_i(t_n,x_n)}{\partial t}-\mathcal{L}^i_x \varphi_i(t_n,x_n)-f(t_n,x_n,i)-\lambda_n\sum_{j\neq i}\exp\left(-\frac{\varepsilon}{2\lambda_n}\right)\nonumber\\
&\geq\frac{\delta}{2}-\lambda_n(m-1)\exp\left(-\frac{\varepsilon}{2\lambda_n}\right)>0.
\end{align}
The inequalities \eqref{eq:sub-2} and \eqref{eq:sub-3} are contradictory. Therefore, we conclude that the assumption \eqref{eq:sub-1} is note true, which implies that  $(\overline{V}_1,\cdots,\overline{V}_m)$ is a viscosity subsolution of system \eqref{eq:HJB-VI}.

We next show that the tuple of lower weak limits $(\underline{V}_1,\cdots,\underline{V}_m)$ is a viscosity supersolution of system \eqref{eq:HJB-VI} by contradiction. For $i\in \mathbb{I}_m$, let  $(t_0,x_0)\in{\cal D}$ and the test function $\varphi_i\in C^{1,2}({\cal D})$ such that $(t_0,x_0)$ is a local minimum of $\overline{V}_i-\varphi_i$. Assume that
\begin{align}\label{eq:min-leq}
\min\Bigg\{&-\frac{\partial \varphi_i(t_0,x_0)}{\partial t}-\mathcal{L}^i_x \varphi_i(t_0,x_0)-f(t_0,x_0,i),\nonumber\\
&\qquad\qquad\quad \underline{V}_i(t_0,x_0)-\max_{j\neq i}(\underline{V}_{j}(t_0,x_0)-g_{ij}) \Bigg\}<0.
\end{align}
Using Lemma 1.6 in Chapter V of \cite{bardi1997optimal} again, there exists a sequence $\{(t_n,x_n)\}_{n \geq 1}$ with $(t_n,x_n)\in {\cal D}$ and   a sequence $\{\lambda_n\}_{n \geq 1}$ with $\lambda_n>0$, $\lim_{n\to \infty}\lambda_n=0$ such that $(t_n,x_n)$ is a local minimum point of $V^{\lambda_n}_i-\varphi_i$ and 
\begin{align}\label{eq:lower-lim-point}
\lim_{n\to \infty}(t_n,x_n)=(t_0,x_0),\quad \lim_{n\to \infty}V^{\lambda_n}_i(t_n,x_n)=\underline{V}_i(t_0,x_0).
\end{align}
By Lemma \ref{thm:sol-HJB}, for any $\lambda>0$, $(V^{\lambda}_1,V^{\lambda}_2,\cdots,V^{\lambda}_m)$ is a classical solution to the system of of HJB equations \eqref {eq:HJB-V}, thus $V^{\lambda_n}_i$ is a viscosity supersolution of the following PDE:
\begin{align*}
-\frac{\partial V^{\lambda_n}_i(t,x)}{\partial t}-\mathcal{L}^i_x V^{\lambda_n}_i(t,x)-f(t,x,i)-\lambda\sum_{j\neq i}\exp\left(\frac{V^{\lambda_n}_{j}(t,x)-g_{ij}-V^{\lambda_n}_{i}(t,x)}{\lambda}\right)=0.
\end{align*}
Therefore we have 
\begin{align}\label{eq:super-1}
-\frac{\partial \varphi_i(t_n,x_n)}{\partial t}-\mathcal{L}^i_x \varphi_i(t_n,x_n)-f(t_n,x_n,i)-\lambda_n\sum_{j\neq i}\exp\left(\frac{V^{\lambda_n}_{j}(t_n,x_n)-g_{ij}-V^{\lambda_n}_{i}(t_n,x_n)}{\lambda}\right)\geq0
\end{align}
for any $n\geq 1$. We consider two cases for the inequality \eqref{eq:min-leq}.

\noindent{\bf Case 1.} Assume that 
\begin{align}\label{eq:super-2}
-\frac{\partial \varphi_i(t_0,x_0)}{\partial t}-\mathcal{L}^i_x \varphi_i(t_0,x_0)-f(t_0,x_0,i)<0.
\end{align}
By \eqref{eq:super-1}, we have
\begin{align*}
-\frac{\partial \varphi_i(t_n,x_n)}{\partial t}-\mathcal{L}^i_x \varphi_i(t_n,x_n)-f(t_n,x_n,i)\geq \lambda_n\sum_{j\neq i}\exp\left(\frac{V^{\lambda_n}_{j}(t_n,x_n)-g_{ij}-V^{\lambda_n}_{i}(t_n,x_n)}{\lambda}\right)\geq0,
\end{align*}
which yields
\begin{align*}
&-\frac{\partial \varphi_i(t_0,x_0)}{\partial t}-\mathcal{L}^i_x \varphi_i(t_0,x_0)-f(t_0,x_0,i)\nonumber\\
&=\lim_{n\to \infty} \left(-\frac{\partial \varphi_i(t_n,x_n)}{\partial t}-\mathcal{L}^i_x \varphi_i(t_n,x_n)-f(t_n,x_n,i)\right)\geq 0.
\end{align*}
Thus, we obtain a contradiction.\\

\noindent{\bf Case 2.} Assume that 
\begin{align}\label{eq:super-3}
\delta:=-\frac{\partial \varphi_i(t_0,x_0)}{\partial t}-\mathcal{L}^i_x \varphi_i(t_0,x_0)-f(t_0,x_0,i)\geq 0,
\end{align}
and 
\begin{align}\label{eq:super-4}
\varepsilon:=-(\underline{V}_i(t_0,x_0)-\max_{j\neq i}(\underline{V}_{j}(t_0,x_0)-g_{ij}))=\underline{V}_{k}(t_0,x_0)-g_{ik}-\underline{V}_i(t_0,x_0)>0.
\end{align}
By \eqref{eq:lower-lim-point}, \eqref{eq:super-3} and \eqref{eq:super-4}, there exists some $n_1>0$ such that for all $n\geq n_1$,
\begin{align*}
-\frac{\partial \varphi_i(t_n,x_n)}{\partial t}-\mathcal{L}^i_x \varphi_i(t_n,x_n)-f(t_n,x_n,i)\leq \frac{3\delta}{2},
\end{align*}
and
\begin{align*}
V^{\lambda_n}_k(t_n,x_n)-g_{ik}-V^{\lambda_n}_{i}(t_n,x_n) \geq \frac{\varepsilon}{2}.
\end{align*}
Selecting $n_2$ such that for all $n\geq n_2$, $\lambda_n\exp(\frac{\varepsilon}{2\lambda_n})>\frac{3\delta}{2}$, then for $n\geq \max\{n_1,n_2\}$, it holds that 
\begin{align}\label{eq:super-5}
&-\frac{\partial \varphi_i(t_n,x_n)}{\partial t}-\mathcal{L}^i_x \varphi_i(t_n,x_n)-f(t_n,x_n,i)-\lambda_n\sum_{j\neq i}\exp\left(\frac{V^{\lambda_n}_{j}(t_n,x_n)-g_{ij}-V^{\lambda_n}_{i}(t_n,x_n)}{\lambda}\right)\nonumber\\
&\leq -\frac{\partial \varphi_i(t_n,x_n)}{\partial t}-\mathcal{L}^i_x \varphi_i(t_n,x_n)-f(t_n,x_n,i)-\lambda_n\exp\left(\frac{V^{\lambda_n}_{k}(t_n,x_n)-g_{ik}-V^{\lambda_n}_{i}(t_n,x_n)}{\lambda}\right)\nonumber\\
&\leq\frac{3\delta}{2}-\lambda_n\exp\left(\frac{\varepsilon}{2\lambda_n}\right)<0.
\end{align}
The inequalities \eqref{eq:super-1} and \eqref{eq:super-5} are contradictory.

Combining the arguments in two cases above, we conclude that assertion \eqref{eq:min-leq} does not hold. This implies that $(\overline{V}_1,\cdots,\overline{V}_m)$ is a viscosity supersolution of system \eqref{eq:HJB-VI}, which completes the proof.
\end{proof}

As the second main result of this paper, the next theorem
shows the convergence result towards the classical optimal switching problem as the entropy regularization vanishes. 
\begin{theorem}\label{thm:convergence-lambda}
Let Assumptions \ref{assump:Lip}, \ref{assump:bound} and \ref{assump:Holder} hold. Consider the value functions $(V_1,\cdots,V_m)$ of the classical optimal switching problem defined by \eqref{eq:value-func-C}, and the value functions $(V_1^{\lambda},\cdots,V_m^{\lambda})$ of the exploratory optimal switching problem defined by \eqref{eq:V}. Then for any $i\in \mathbb{I}_m$ and $(t,x)\in\overline{{\cal D}}$,
\begin{align}
\lim_{\lambda\to 0}V_i^{\lambda}(t,x)=V_i(t,x).
\end{align}
\end{theorem}
\begin{proof}
By using Lemma \ref{lem:uper-lower} and Lemma \ref{lem:comparison}, we have 
\begin{align*}
\overline{V}_i(t,x)\leq \underline{V}_i(t,x),\quad \forall i\in\mathbb{I}_m,(t,x)\in\overline{{\cal D}}.
\end{align*}
On the other hand, it follows from the definition of upper and lower weak limits that $\overline{V}_i(t,x)\geq \underline{V}_i(t,x)$, for any $i\in\mathbb{I}_m$ and $(t,x)\in\overline{{\cal D}}$. Thus,  $\overline{V}_i(t,x)=\underline{V}_i(t,x)$, and we denote by
\begin{align*}
V^*_i(t,x):=\overline{V}_i(t,x)=\underline{V}_i(t,x)\quad \text{for}~ i\in\mathbb{I}_m,(t,x)\in\overline{{\cal D}}.
\end{align*}
It follows from \eqref{eq:upper}, \eqref{eq:lower} and  Lemma \ref{lem:uper-lower} that $(V_1^*,\cdots,V_m^*)$ is a bounded viscosity solution of system \eqref{eq:HJB-VI} satisfying $V^*_i(t,x)=\lim_{\lambda\to 0}V_i^{\lambda}(t,x)$. We deduce from Theorem \ref{thm:viscosity} that
\begin{align}
V_i(t,x)=V^*_i(t,x)=\lim_{\lambda\to 0}V_i^{\lambda}(t,x).
\end{align}
Thus, we complete the proof of the theorem.
\end{proof}
Theorem \ref{thm:convergence-lambda} justifies the use of the exploratory formulation as a well-founded mathematical relaxation: as the exploration effect diminishes (as the temperature parameter $\lambda \to 0$), the value function of the exploratory formulation indeed converges towards the value function of the classical optimal switching problem. Mathematically speaking, it is interesting to observe that the solution to the system of PDEs will converge to the solution of system of variational inequalities. Therefore, our exploratory formulation can also be regarded as a penalization approach to study a system of variational inequalities, under which we only need to handle the existence and regularity of solution to a system of PDEs.

\section{Reinforcement Learning Algorithm}\label{sec:algorithm}
In this section, we design a RL algorithm to solve the exploratory optimal switching problem. We assume that the  the drift $\mu$, volatility $\sigma$, running profit $f$,  and terminal reward $h$ are all assumed unknown; whereas the switching costs $g_{ij}$ are known constants and observable. The core of our approach lies in a key reformulation: we have transformed the original optimal switching problem into a standard optimal control problem where we control the generator of the finite-state Markov chain that characterizes the switching regimes. The primary distinction from classical problems is that the agent now actively controls the transition rates between regimes, adding a continuous layer of decision-making on top of the discrete switching choices.

Our choice of the randomization and the exploratory form  leads to an explicit characterization of the optimal policy that depends on the value functions, without involving their derivatives.  Leveraging this solution structure, we adopt the policy evaluation (PE) method based on the martingale characterization method similar to \cite{jia2022policy}, which consider two alternative methods based on a martingale characterization: minimizing a martingale loss function, which provides the best mean-square approximation of the true value function, and solving a system of martingale orthogonality condition with test functions. In what follows, we design the PE algorithm by the martingale orthogonality condition and  the established policy improvement result in Proposition \ref{thm:improvement}.

Recall that given a feedback strategy ${\bm \pi}(t,x)=(\pi_{ij}(t,x))_{i,j\in \mathbb{I}_m}$, the corresponding value function $(v_1^{{\bm \pi}},\cdots,v_m^{{\bm \pi}})$ satisfies the PDE system that for $i\in \mathbb{I}_m$,
\begin{align}\label{eq:V-pi}
\begin{cases}
 \displaystyle  \frac{\partial v_i^{{\bm \pi}}(t,x)}{\partial t}+\mathcal{L}^i_x v_i^{{\bm \pi}}(t,x)+f(t,x,i)+H_i({\bm \pi}_i(t,x),v_1^{{\bm \pi}}(t,x),\cdots,v^{{\bm \pi}}_m(t,x))=0,\\
 \displaystyle v^{{\bm \pi}}_i(T,x)=h(x),
\end{cases}
\end{align}
where the Hamiltomian $H_i$ is given by \eqref{eq:H}. For simplicity, we omit the superscript ${{\bm \pi}}$ and denote the value function as:
\begin{align}\label{eq:v-i}
v(t,x,i) = v^{{\bm \pi}}_i(t,x), \quad \text{for } i \in \mathbb{I}_m, ~(t,x) \in \overline{{\cal D}},
\end{align}
and denote by $I = (I_t)_{t \geq 0}$ a continuous-time finite-state Markov chain with generator ${\bm \pi} = (\pi^{ij})_{i,j \in \mathbb{I}_m}$. Let us introduce the process $M=(M_t)_{t\in[0,T]}$ given by
\begin{align}\label{eq:M}
M_t:=v(t,X_t,I_t)+\int_0^t\left(f(s,X_{s},I_s)+\lambda R({\bm \pi}(s,X_s),I_s)\right)ds-\sum_{k=1}^\infty g_{\kappa_{k-1}\kappa_k}{\bf 1}_{\{\tau_k\leq t\}},\quad t\in[0,T].
\end{align}
The next lemma gives the martingale characterization that lays the foundation for the loss function and the policy evaluation RL algorithm. 
\begin{lemma}\label{lem:martingale}
Let  ${\bm \pi}(t,x)=(\pi_{ij}(t,x))_{i,j\in \mathbb{I}_m}$ be a feedback strategy and $v(t,x,i)$ be the corresponding value function given by \eqref{eq:v-i}. Then the process $M=(M_t)_{t\in[0,T]}$ given by \eqref{eq:M} is a square-integrable martingale.
\end{lemma}
\begin{proof}
Using It\^o's rule to $v(s,X_s,I_s)$ from $t'$ to $t$, we obtain
\begin{align}\label{eq:v}
&v(t,X_t,I_t)\nonumber\\
&=v(t',X_{t'},I_{t'})+\int_{t'}^t (D_x v(s,X_s,I_s))^{\top}\sigma(s,X_{s},I_s)dW_{s}\nonumber\\
&\qquad+
\int_{t'}^t \int_{0}^{\infty}\left(v(s,X_{s-},I_{s-}+p(s,X_{s-}, I_{s-},z) )-v(\ell,X_{s-},I_{s}-)\right) \tilde{{\cal N}}(d s, d z)\nonumber\\
&\qquad+\int_{t'}^t \left(\frac{\partial v(s,X_s,I_s)}{\partial t}+\mathcal{L}^{I_s}_x v(s,X_{s},I_s)+\sum_{j\neq I_s}\left(\pi^{I_sj}(s,X_s)(v(s,X_{s},j)-v(s,X_{s},I_s))\right)\right)ds.
\end{align}
It follows from  \eqref{eq:V-pi}, \eqref{eq:M} and \eqref{eq:v} that
\begin{align*}
&\Ex[M_t|\F_{t'}]\nonumber\\
&=\Ex\left[v(t,X_t,I_t)+\int_0^t\left(f(s,X_{s},I_s)+\lambda R({\bm \pi}(s,X_s),I_s)\right)ds-\sum_{k=1}^\infty g_{\kappa_{k-1}\kappa_k}{\bf 1}_{\{\tau_k\leq t\}}\Big|\F_{t'}\right]\nonumber\\
&=M_{t'}+\Ex\left[v(t,X_t,I_t)-v(t',X_{t'},I_{t'})\Big|\F_{t'}\right]\nonumber\\
&\quad+\Ex\left[\int_{t'}^t\left(f(s,X_{s},I_s)+\lambda R({\bm \pi}(s,X_s),I_s)\right)ds-\int_{t'}^t \sum_{j\neq I_s} g_{I_s j}\pi_s^{I_sj}ds\Big|\F_{t'}\right]\nonumber\\
&=M_{t'}+\Ex\left[\int_{t'}^t (D_x v(s,X_s,I_s))^{\top}\sigma(s,X_{s},I_s)dW_{s}\Big|\F_{t'}\right]\nonumber\\
&\quad+\Ex\left[\int_{t'}^t \int_{0}^{\infty}\left(v(s,X_{s-},I_{s-}+p(s,X_{s-}, I_{s-},z) )-v(s,X_{s-},I_{s-})\right) \tilde{N}(d s, d z)\Big|\F_{t'}\right]\nonumber\\
&=M_{t'}.
\end{align*}
Thus, we get the desired result.
\end{proof}

Let us introduce the notation $L^2([0,T])$ as the space of all processes $K = (K_t)_{t\in[0,T]}$ that $K$ is $\mathbb{F}$-progressively measurable satisfying $\Ex[\int_0^T |K_t|^2 dt]<\infty$ and $\Ex[\left< K\right>_T]<\infty$.  For any semimartingale $R=(R_s)_{s\in[0,T]}$, we denote $L^2([0,T];R)$ the space of all processes $K = (K_t)_{t\in[0,T]}$ that $K$ is $\mathbb{F}$-progressively measurable and satisfies
\begin{align*}
\mathbb{E}\left[\int_0^T |K_t|^2 d\left<R\right>_t\right] < \infty,
\end{align*}
where $\left<R\right>_t$ is the quadratic variation process of $R$. We have the following result.

\begin{proposition}\label{prop:orthogonality}
A diffusion process $R=(R_s)_{s\in[0,T]}\in L^2([0,T])$ is a martingale if and only if
\begin{align}\label{eq:orthogonality}
\Ex\left[\int_0^{T}\varsigma_tdR_t\right]=0
\end{align}
 for any $\varsigma\in L^2([0,T];R)$.
\end{proposition}
\begin{proof}
As $R$ is a diffusion process, it admits the stochastic differential representation
\begin{align*}
R_t=R_0+\int_0^t A_sds+\int_0^t M_sdB_s,\quad t\in[0,T],
\end{align*}
where $B = (B_t)_{t\in[0,T]}$ is a $\mathbb{F}$-adapted scalar Brownian motion, and  $A=(A_t)_{t\in[0,T]}$ and $M=(M_t)_{t\in[0,T]}$ are $\mathbb{F}$-progressively measurable processes. If $R$ is a martingale, then its drift vanishes, i.e., $A\equiv0$. Now take any  $\varsigma\in L^2([0,T];R)$. Then
\begin{align*}
\mathbb{E}\left[\int_0^T |\varsigma_tM_t|^2 dt\right]=\mathbb{E}\left[\int_0^T |\varsigma_t|^2 d\left<R\right>_t\right]< \infty.
\end{align*}
This implies that $(\int_0^t \varsigma_sM_sdB_s)_{t\in[0,T]}$ is a martingale; consequently,
\begin{align*}
\Ex\left[\int_0^{T}\varsigma_tdR_t\right]=\Ex\left[\int_0^{T}\varsigma_tM_tdB_t\right]=0.
\end{align*}
Conversely, assume that \eqref{eq:orthogonality} holds for every $\varsigma\in L^2([0,T];R)$. For any $t,t'$ with $0\leq t\leq t'\leq T$, define
\begin{align*}
\varsigma_s = \begin{cases}
\operatorname{sgn}(A_s), & s \in [t, t'], \\
0, & \text{otherwise}.
\end{cases}
\end{align*}
 As $|\varsigma|\leq 1$ for all $t\in[0,T]$, we have $\varsigma\in L^2([0,T];R)$. Then
\begin{align*}
0=\Ex\left[\int_0^{T}\varsigma_sdR_s\right]=\Ex\left[\int_t^{t'}|A_s|ds\right]+\Ex\left[\int_t^{t'}\varsigma_sM_sdB_s\right]=\Ex\left[\int_t^{t'}|A_s|ds\right],
\end{align*}
where the last equality follows because $(\int_t^\ell \varsigma_sM_sdB_s)_{\ell\in[t,t']}$ is a martingale. By the arbitrariness of $t$ and $t'$, we conclude that $A_s=0$ a.s. for every $s\in[0,T]$. Hence $N$ is a martingale, which completes the proof.
\end{proof}
Proposition \ref{prop:orthogonality} is known as the martingale orthogonality condition (see, e.g., \cite{jia2022policy,jia2023q}). It provides a powerful characterization of martingales within the class of diffusion processes: a process is a martingale if and only if its stochastic integral against every admissible test process has zero expectation. This condition serves the theoretical backbone of our policy evaluation RL algorithm. 

Given a feedback strategy ${\bm \pi}(t,x)=(\pi_{ij}(t,x))_{i,j\in \mathbb{I}_m}$, we parameterize the value function by a neural network $v^\xi(t,x,i)$ that is differentiable with respect to the parameter vector $\xi$ and satisfies $v^{\xi}(T,x,i)=h(x)$. Here $\xi \in \Theta \subset \mathbb{R}^{L_\xi}$ and $L_{\xi}$ is the dimension of the parameter vector.  The term $\frac{\partial v^{\xi}}{\partial \xi}$ denotes the gradient of the neural network output with respect to its parameters.

Let $M^{\xi}=(M_t^{\xi})_{t\in[0,T]}$ be the parameterized version of the martingale process $M$ defined in \eqref{eq:M}.  Our goal is to find the optimal parameters $\xi$ such that $M^\xi$ is as close as possible to a true martingale. Ideally, we would like $M^\xi$ to be an exact martingale; then, by Proposition~\ref{prop:orthogonality}, the orthogonality condition
\begin{align}\label{eq:orthogonality-M}
\Ex\left[\int_0^{T}\varsigma_tdM^{\xi}_t\right]=0
\end{align}
must hold for every test process $\varsigma\in L^2([0,T];M)$. Equation \eqref{eq:orthogonality-M} can be viewed as a (infinite-dimensional) family of equations to determine the parameter $\xi$. We refer to this as the martingale orthogonality equation. In principle, solving the full set of equations would require verifying the condition for all $\varsigma$, which is numerically infeasible. For practical approximation, we select a finite collection of test processes that are sufficiently rich to determine $\xi$. Because we have $L_\xi$ unknown parameters, we need at least $L_\xi$ independent equations. A natural choice is to take the test process proportional to the gradient of the approximate value function with respect to the parameters, namely $\varsigma_t = \frac{\partial v^{\xi}}{\partial \xi}(t,X_t,I_t)$.  

To solve the martingale orthogonality equations, we  apply the stochastic approximation method (see, e.g., \cite{robbins1951stochastic,jia2022policy}) --- a class of iterative algorithms for finding roots of equations or optimizing objectives when only noisy measurements are available. Substituting the chosen test process 
$\varsigma=\frac{\partial v^{\xi}}{\partial \xi}(t,X_t,I_t)$ into the orthogonality condition and applying  stochastic approximation yields the  updating rule
\begin{align}\label{eq:SA}
 \xi \leftarrow \xi+\alpha_\xi \int_0^{T} \frac{\partial v^{\xi}}{\partial \xi}(s,X_s,I_s)dM_s^{\xi},
\end{align}
where $\alpha_\xi>0$ is the learning rate. 

However, the update rule \eqref{eq:SA} involves a continuous-time integral that cannot be directly implemented computationally. To address this, we adopt a discrete-time approximation of the martingale orthogonality condition. Let $K\in\mathbb{N}$ be the number of time intervals and $\Delta t=T/K$ be the step size. Consider the discretization grid $0=t_0<t_1<t_2<\cdots<t_K=T$ with $t_k-t_{k-1}=\Delta t$ for $k=1,\dots,K$. We have
\begin{align*}
\int_0^T \varsigma_t \, dM_t^\xi \approx \sum_{k=0}^{K-1} \varsigma_{t_k} \, \Delta M_{t_k}^\xi,
\end{align*}
where \(\Delta M_{t_k}^\xi = M_{t_{k+1}}^\xi - M_{t_k}^\xi\). Expanding \(\Delta M_{t_k}^\xi\) by the definition of \(M^\xi\) gives
\begin{align}\label{eq:delta-xi}
&\Delta M_{t_k}^\xi = M_{t_{k+1}}^\xi - M_{t_k}^\xi\nonumber\\
&=v^{\xi}(t_{k+1},X_{t_{k+1}},I_{t_{k+1}})-v^{\xi}(t_k,X_{t_k},I_{t_k})+\int_{t_k}^{t_{k+1}}\left(f(s,X_{s},I_{s})+\lambda R({\bm \pi}^{\xi}_{s},I_{s})\right)ds\nonumber\\
&\qquad-\sum_{j=1}^\infty g_{\kappa_{j-1}\kappa_j}{\bf 1}_{\{t_k<\tau_j\leq t_{k+1}\}}\nonumber\\
&\approx v^{\xi}(t_{k+1},X_{t_{k+1}},I_{t_{k+1}})-v^{\xi}(t_k,X_{t_k},I_{t_k})+\left(f(t_k,X_{t_k},I_{t_k})+\lambda R({\bm \pi}^{\xi}_{t_k},I_{t_k})\right)\Delta t- g_{I_{t_{k}}I_{t_{k+1}}}\nonumber\\
&:=\Delta \xi_k,
\end{align}
where the parameterized  strategy ${\bm \pi}^{\xi}(t,x)=(\pi^{\xi}_{ij}(t,x))_{i,j\in \mathbb{I}_m}$ is given by
\begin{align}\label{eq:policy-xi}
    \pi_{ij}^{\xi}(t,x)=\exp\left(\frac{v^\xi(t,x,j)-g_{ij}-v^{\xi}(t,x,i)}{\lambda}\right),  ~j\neq i,
\end{align}
and $\pi_{ii}^{\xi}(t,x)=-\sum_{j\neq i}\pi_{ij}^{\xi}(t,x)$. Here, we choose the time step $\Delta t$ to be sufficiently small relative to the transition rates of the continuous-time Markov chain $I_t$. Under this assumption, the probability of two or more jumps occurring within any single interval $(t_k, t_{k+1}]$ is of order $O((\Delta t)^2)$ and is hence negligible for small $\Delta t$. That is, we can assume that at most one regime switch occurs during each time step. Thus, we replace the infinite sum over jump times by the single switching cost $g_{I_{t_k} I_{t_{k+1}}}$ in \eqref{eq:delta-xi}.
Substituting this into the discretized integral yields the parameter update
\begin{align}\label{eq:offline}
\xi \leftarrow \xi + \alpha_\xi \sum_{k=0}^{K-1} \varsigma_{t_k} \, \Delta \xi_k.
\end{align}
Note that this is an offline update: the entire trajectory over $[0,T]$ is collected first, all $\Delta \xi_k$ are computed, and then the parameter is updated once per episode. Alternatively, one can update the parameter in an online fashion at each time step as soon as $\Delta \xi_k$ becomes available:
\begin{align*}
\xi \leftarrow \xi + \alpha_\xi \, \varsigma_{t_k} \, \Delta \xi_k.
\end{align*}
The online update does not wait for the end of the episode; it uses data incrementally, which has better data efficiency, albeit it may introduce more noise per update.


Based on the above updating rules, we can present the pseudo-code of the offline PE algorithm in Algorithm \ref{Alg:PE}. The online PE algorithm can be devised in a similar fashion and is omitted.

	\begin{algorithm}[h]
			\caption{\textbf{Policy Evaluation Algorithm (Offline)}}
		    \label{Alg:PE}
			\hspace*{0.02in} {\bf Input:} 
			 Initial state $(x_0,i_0)$, horizon $T$, number of regimes $m$, time step $\Delta t$, number of episodes $N$, number of mesh grids $K$, initial learning rates $\alpha_\xi(\cdot)$  (a function of the number of episodes), functional forms of parameterized value function $v^\xi(\cdot)$, policy $\boldsymbol{\pi}^\xi(\cdot)$, regime switching costs $(g_{ij})_{i,j\in \mathbb{I}_m}$ and temperature parameter $\lambda$.\\
			\hspace*{0.02in} {\bf Required Program:} an environment simulator $(x^{\prime}, i^{\prime},f^{\prime})=$ Environment $_{\Delta t}(t,x,i,j)$ that takes current time-state pair $(t,x,i)$ and action $j$ (the regime to switch to; if $j = i$, no switching occurs) as inputs and generates state $x^{\prime}$, $i^{\prime}=j$ and reward $f'$ at time $t+\Delta t$ as outputs . \\
			\hspace*{0.02in} {\bf Learning Procedure:}
			\begin{algorithmic}[1]
\State Initialize $\xi$, and $\ell=1$. 
				\While{$\ell<N$}  
				\State Initialize $k = 0$.  Observe initial state $x_0,i_0$ and store $(x_{t_0},i_{t_0}) \leftarrow (x_0,i_0)$.
				\While{$k < K$}
				  
				    \State \ \ \ \ \    Generate action $j_{t_{k}}$ by $\bm{\pi}^\xi\left(t_{k},x_{t_{k}}\right)$.
                   \State \ \ \ \ \   Apply $j_{t_{k}}$ to environment simulator $(x, i,f)=$ Environment  $_{\Delta t}(t_k, x_{t_k},i_{t_k},j_{t_{k}})$.
                   \State \ \ \ \ \   Observe  new state $x$ and $i$ as output. Store $x_{t_{k+1}} \leftarrow x$,  $ i_{t_{k+1}} \leftarrow  i$ and $f_{t_k} \leftarrow  f$. 
                   \State \ \ \ \ \  Update  $k \leftarrow k+1$. 
                \EndWhile{} 
                \State  For every $k=0,1,...,K-1$, compute 
\begin{align*}
 \Delta \xi_{k}&=v^{\xi}(t_{k+1},x_{t_{k+1}},i_{t_{k+1}})-v^{\xi}(t_k,x_{t_k},i_{t_k})+\left(f_{t_k}+\lambda R({\bm \pi}^{\xi}(t_k,x_{t_k}),i_{t_k})\right)\Delta t- g_{i_{t_{k}}i_{t_{k+1}}}.
\end{align*}

				\State Update $\xi$ by
\begin{align*}
\xi &\leftarrow \xi+\alpha_\xi(\ell+1) \sum_{k=0}^{K-1}  \frac{\partial v^\xi}{\partial \xi}\left(t_k,x_{t_k},i_{t_k}\right)\Delta \xi_{k}, \\
\end{align*}
               \State   Update $\ell \leftarrow \ell+1$.
				          
              \EndWhile{}                
    \end{algorithmic}
\end{algorithm}

Proposition \ref{thm:improvement} and Theorem \ref{thm:convergence} confirm the improvement and convergence results of the policy iteration. Meanwhile, Lemma \ref{lem:martingale} and Proposition \ref{prop:orthogonality} show that policy evaluation can be performed by solving the martingale orthogonality condition via stochastic approximation. A natural question arises: what can be said about the convergence of Algorithm \ref{Alg:PE}? To address this, we next turn to an analysis of the error estimates for Algorithm \ref{Alg:PE}.

We reformulate the update rule in equation \eqref{eq:offline} as follows:
\begin{align}\label{eq:offline-ell}
\xi_{\ell+1} \leftarrow \xi_\ell + \alpha_\xi(\ell+1) \Psi(\xi_\ell; X, I, {\bm \pi}^{\xi_\ell}), \quad \text{for iteration step}~\ell \geq 1,
\end{align}
where
\begin{align*}
\Psi(\xi_\ell; X, I, {\bm \pi}^{\xi_\ell}) = \sum_{k=0}^{K-1} \frac{\partial v^{\xi}}{\partial \xi}(t_k, X_{t_k}, I_{t_k}) \Delta \xi_k,
\end{align*}
with $\Delta \xi_k$ defined in equation \eqref{eq:delta-xi}. For notational convenience, we introduce the shorthand $Y_{\ell+1} = (X, I, {\bm \pi}^{\xi_\ell})$ for $\ell \geq 1$. We further define the expected update function as $\psi(\xi) := \mathbb{E}[\Psi(\xi; Y)]$. To establish convergence guarantees, we make the following technical assumptions.

\begin{assumption}\label{assump:convergence}
\begin{itemize}
\item[(i)] The ordinary differential equation $\xi'(t) = \psi(\xi(t))$ has a unique stable equilibrium point $\xi^*$.
\item[(ii)] There exists a constant $C > 0$ such that $\mathbb{E}[|\Psi(\xi_{\ell}; Y_{\ell+1})|^2 | \xi_\ell] \leq C(1 + |\xi_\ell|^2)$ for all iterations.
\item[(iii)] There exists $\kappa > 0$ such that $(\xi - \xi^*) \cdot \psi(\xi) \leq -\kappa|\xi - \xi^*|^2$ for all $\xi \in \mathbb{R}^{L_\xi}$.
\item[(iv)] There exist constants $\rho, C > 0$ such that $\sup_{j\in\mathbb{I}_m}|v^{\xi}(\cdot,j)- v^{\xi^*}(\cdot,j)|_{C^0(\overline{{\cal D}})} \leq C|\xi - \xi^*|^{\rho}$ for all $\xi \in \mathbb{R}^{L_\xi}$.
\end{itemize}
\end{assumption}

Under these conditions, we now present the main convergence result, which provides the explicit error bound for Algorithm \ref{Alg:PE}.
\begin{theorem}\label{thm:error}
Let Assumption \ref{assump:convergence} hold. Set $\alpha_\xi(\ell)=\frac{A}{\ell^{\nu}+B}$ for some $\nu\in(0, 1]$, $A>\frac{\nu}{2\kappa}$ and $B>0$, and let $\epsilon>0$. Then there exists $C>0$ (independent of $n,\epsilon$) such that with probability of at least $1-\epsilon$, 
\begin{align}\label{eq:error}
\sup_{j\in\mathbb{I}_m}|v^{\xi_\ell}(\cdot,j)-v(\cdot,j)|_{C^0({\overline{{\cal D}}})}\leq \sup_{j\in\mathbb{I}_m}|v(\cdot,j)-v^{\xi^*}(\cdot,j)|_{C^0({\overline{{\cal D}}})}+\frac{C}{\epsilon^{\rho/2}}\ell^{-\frac{\nu\rho}{2}}.
\end{align}
\end{theorem}
\begin{proof}
It follows from \eqref{eq:offline-ell} and Assumption \ref{assump:convergence}-(ii),(iii) that
\begin{align}\label{SA-1}
\Ex\left[|\xi_{\ell+1}-\xi^*|^2 \mid \xi_\ell\right] 
&=\Ex\left[|\xi_{\ell}-\xi^*-\xi_{\ell}+\xi_{\ell+1}|^2 \mid \xi_\ell\right]\nonumber\\
&\leq|\xi_{\ell}-\xi^*|^2 +2|\xi_{\ell}-\xi^*|\Ex\left[|\xi_{\ell}-\xi_{\ell+1}| \mid \xi_\ell\right]+\Ex\left[|\xi_{\ell}-\xi_{\ell+1}|^2 \mid \xi_\ell\right]\nonumber\\
&=|\xi_{\ell}-\xi^*|^2+2\alpha_\xi(\ell+1)(\xi_{\ell}-\xi^*)\psi(\xi)+\mathbb{E}[|\Psi(\xi_{\ell}; Y_{\ell+1})|^2\mid\xi_{\ell}] \nonumber\\
&\leq  (1-2\alpha_\xi(\ell+1)\kappa)|\xi_{\ell} - \xi^*|^2+C\alpha^2_\xi(\ell+1)(1 + |\xi_\ell|^2)\nonumber\\
  &\leq (1-2\alpha_\xi(\ell+1)\kappa)|\xi_{\ell} - \xi^*|^2+C\alpha^2_\xi(\ell+1)(1 +2 |\xi_\ell-\xi^*|^2+2|\xi^*|^2)\nonumber\\
&\leq (1-2\alpha_\xi(\ell+1)\kappa+2C\alpha^2_\xi(\ell+1))|\xi_{\ell} - \xi^*|^2+C\alpha^2_\xi(\ell),
\end{align}
where \(C>0\) is a generic constant independent of \(n,\epsilon\), which may differ from line to line.

Taking expectations on both sides of \eqref{SA-1}, we obtain
\begin{align}\label{SA-2}
\Ex\left[|\xi_{\ell+1}-\xi^*|^2\right] \leq (1-2\alpha_\xi(\ell+1)\kappa+2C\alpha^2_\xi(\ell+1))\Ex\left[|\xi_{\ell} - \xi^*|^2\right]+C\alpha^2_\xi(\ell).
\end{align}
As \(\alpha_\xi(\ell)=\frac{A}{\ell^{\nu}+B}\) with \(\nu\in(0, 1]\), \(A>\frac{\nu}{2\kappa}\), and \(B>0\), it follows from Lemma 23 in Chapter 1 of Part II of  \cite{benveniste2012adaptive} that there exist constants \(C_0>0\) and \(\ell_0\geq 1\) such that, for any \(c\geq C_0\) and \(\ell\geq \ell_0\), we have \(1-2\alpha_\xi(\ell+1)\kappa\geq 0\) and
\begin{align}\label{SA-3}
\tilde{\alpha}(\ell+1)\geq (1-2\alpha_\xi(\ell+1)\kappa+2C\alpha^2_\xi(\ell+1))\tilde{\alpha}(\ell)+C\alpha^2_\xi(\ell),
\end{align}
where \(\tilde{\alpha}(\ell)=c \alpha(\ell)\) for \(\ell\geq 1\).

Choose \(c\geq C_0\) such that \(\Ex\left[|\xi_{\ell}-\xi^*|^2\right]\leq c\alpha_\xi(\ell)\) for all \(\ell\leq \ell_0\). For \(\ell\geq \ell_0\), applying \eqref{SA-2} and \eqref{SA-3} and arguing by induction on \(\ell\), we get immediately that \(\Ex\left[|\xi_{\ell}-\xi^*|^2\right]\leq c\alpha_\xi(\ell)\). Thus, it holds that
\begin{align*}
\Ex[|\xi_\ell-\xi^*|^2]\leq C\ell^{-\nu},
\end{align*}
where \(C > 0\) is a constant independent of \(\ell\). This bound in turn implies that
 \begin{align*}
|\xi_\ell-\xi^*|\leq C\epsilon^{-\frac{1}{2}}\ell^{-\frac{\nu}{2}}
\end{align*}
with probability at least $1 - \epsilon$. Then, invoking Assumption \ref{assump:convergence}-(iv), we deduce that with probability at least $1 - \epsilon$,
\begin{align*}
&\sup_{j\in\mathbb{I}_m}|v^{\xi_\ell}(\cdot,j)-v(\cdot,j)|_{C^0({\overline{{\cal D}}})}\\
&\leq \sup_{j\in\mathbb{I}_m}|v(\cdot,j)-v^{\xi^*}(\cdot,j)|_{C^0({\overline{{\cal D}}})}+\sup_{j\in\mathbb{I}_m}|v^{\xi^*}(\cdot,j)-v^{\xi_\ell}(\cdot,j)|_{C^0({\overline{{\cal D}}})}\\
&\leq \sup_{j\in\mathbb{I}_m}|v^{\xi^*}(\cdot,j)-v(\cdot,j)|_{C^0({\overline{{\cal D}}})}+\frac{C}{\epsilon^{\rho/2}}\ell^{-\frac{\nu\rho}{2}}.
\end{align*}
This completes the proof. 
\end{proof}

Theorem \ref{thm:error} establishes a comprehensive error analysis for Algorithm \ref{Alg:PE}, providing both theoretical guarantees and practical insights into its convergence behavior.  Assumption \ref{assump:convergence}-(i) ensures that the unique stable equilibrium $\xi^*$ is the only possible candidate for convergence of the stochastic approximation algorithm (c.f. Theorem 7 in Chapter of Part I of \cite{benveniste2012adaptive}). Assumptions~\ref{assump:convergence}-(ii) and (iii), together with the specific structure of the learning rate in Theorem~\ref{thm:error}, are crucial for characterizing the convergence rate of the stochastic approximation. Assumption~\ref{assump:convergence}-(iv) imposes a Hölder-type regularity condition on the function approximator, connecting the parameter-space error with the function-space error.

The result demonstrates that the policy evaluation error can be systematically decomposed into two distinct components: the approximation error of the parametric function class and the algorithmic error arising from the stochastic approximation procedure. The first term, $\sup_{j\in\mathbb{I}_m}|v(\cdot,j)-v^{\xi^*}(\cdot,j)|_{C^0({\overline{{\cal D}}})}$, represents the inherent approximation capability of our chosen parametric family. This bias term is independent of the learning algorithm and reflects how well the optimal parameter $\xi^*$ can approximate the true value function within the selected function class. The second term, $C\ell^{-\frac{\nu\rho}{2}}/\epsilon^{\rho/2}$, exhibits a polynomial decay with respect to the iteration number $i$ and vanishes asymptotically as the number of iterations increases, demonstrating the algorithm's convergence to the optimal parameter configuration within the chosen function class.

\section{Numerical Examples}\label{sec:numerical}
This section presents some numerical experiments to demonstrate the practical efficacy of the proposed RL algorithm. We first examine a bounded regulator problem to analyze the algorithm's convergence property and policy behavior. Subsequently, we apply the algorithm to a put option selection problem involving the optimal switching between risky assets, showcasing its effectiveness in a more complex, multi-regime setting with some financial interpretations.

\subsection{Bounded Regulator Problem}
To establish a performance benchmark for our algorithm, we consider a finite-horizon optimal switching problem with two regimes, conceptualized as a bounded regulator. This classic problem provides a tractable yet non-trivial testbed where the optimal policy has an intuitive structure, allowing for clear interpretation of the algorithm's learned strategy.

The system state $X = (X_t)_{t \in [0, T]}$ evolves according to regime-specific stochastic dynamics:
\begin{align}\label{eq:example1-X}
dX_t = \mu_i dt + \sigma dW_t, \quad i \in \{0, 1\}, ~ t \in [0, T],
\end{align}
with initial condition $X_0 = x \in \mathbb{R}$. Here, $W = (W_t)_{t \in [0, T]}$ is a standard Brownian motion. The parameters are chosen with symmetry: the drift coefficients are $\mu_0 = -2$ and $\mu_1 = 2$, and the volatility is $\sigma = 0.5$. This symmetric setup induces a natural switching logic to correct the state's deviation.

The controller's objective is to maximize the expected total reward over the horizon $[0, T]$, which comprises a running reward and a terminal reward:
\begin{align*}
f(x) = 2e^{-2x^2} - 0.1, \quad h(x) = 2e^{-2x^2}, \quad x \in \R.
\end{align*}
The Gaussian bump shape of the functions $f$ and $h$ creates a strong incentive to maintain the state $X_t$ near zero, as the reward attains its maximum value at $x=0$. Each switch between regimes incurs a cost, specified as $g_{01} = g_{10} = 0.5$. This cost penalizes excessive control actions, forcing the optimal policy to strategically balance the benefit of corrective switching against the incurred cost.

We use a discrete version of \eqref{eq:example1-X} for $t = 0, \Delta t,..., K\Delta t$ with $T=1$, $K = 100$ and $\Delta t = T/K$. The value function and policy are approximated by a neural network in the PyTorch framework
with the architecture and parameters summarized in Table~\ref{tab:nn_params_regulator}. 

In order to approximate the expectation operator in the theoretical formulation via Monte Carlo methods—that is, to estimate the gradient direction using the average of finite samples so that the update of parameters $\xi$ is not based on a single specific scenario—we employ batch learning to train the neural network parameters $\xi$. This approach not only enhances training stability but also ensures that the learned policy is optimal in an expected sense.

\begin{table}[htbp]
\centering
\caption{Neural Network Architecture and Training Parameters for the Regulator Problem}
\label{tab:nn_params_regulator}
\begin{tabular}{lc}
\toprule
\textbf{Component} & \textbf{Specification} \\
\midrule
Network Architecture & 2 hidden layers \\
Activation Functions & ReLU (Layer 1), Tanh (Layer 2) \\
Hidden Dimension & 128 \\
Batch Size & 64 \\
Optimizer & Adam \\
Learning Rate & \( 1 \times 10^{-3} \) \\
Training Episodes & 1000 \\
\bottomrule
\end{tabular}
\end{table}

The training progression under the temperature parameter  $\lambda = 0.2 $ is shown in Figure~\ref{fig:convergence}-(a). The loss function decreases efficiently and stabilizes after approximately 400 episodes, indicating the robust convergence of the policy iteration in the RL algorithm. Figure~\ref{fig:convergence}-(b) depicts the learned value functions and the optimal switching probabilities at $t = 0.5$. The near symmetry between the value functions for regime 0 (blue line) and regime 1 (orange line) is a direct consequence of the symmetric problem parameters. The switching probabilities---from regime 0 to 1 (green line) and from regime 1 to 0 (yellow line)---are calculated from the optimal intensity $\pi$. 

\begin{figure}[ht]
  \centering
   \subfigure[]{
        \includegraphics[width=0.45\textwidth]{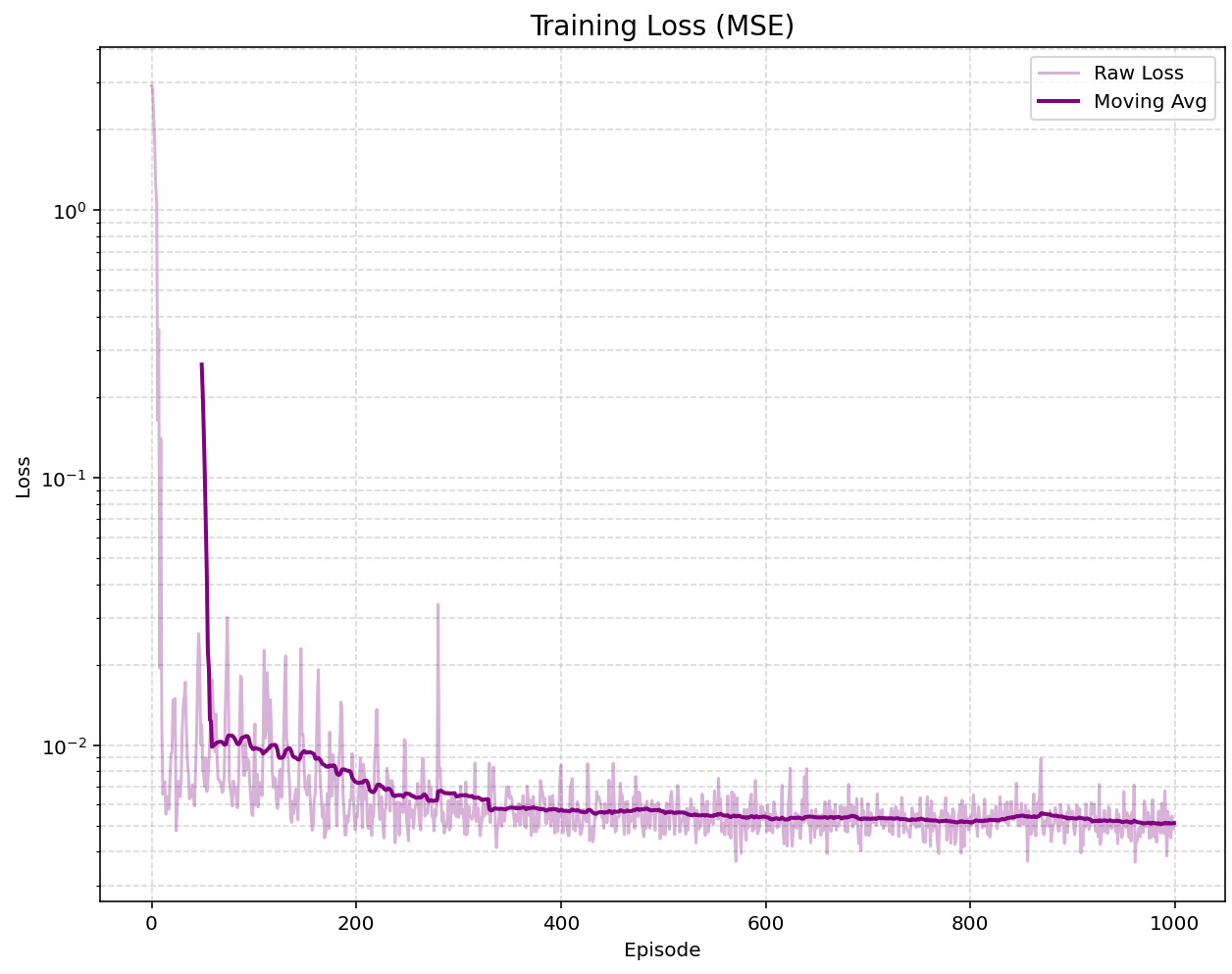}
    }
       \subfigure[]{
        \includegraphics[width=0.45\textwidth]{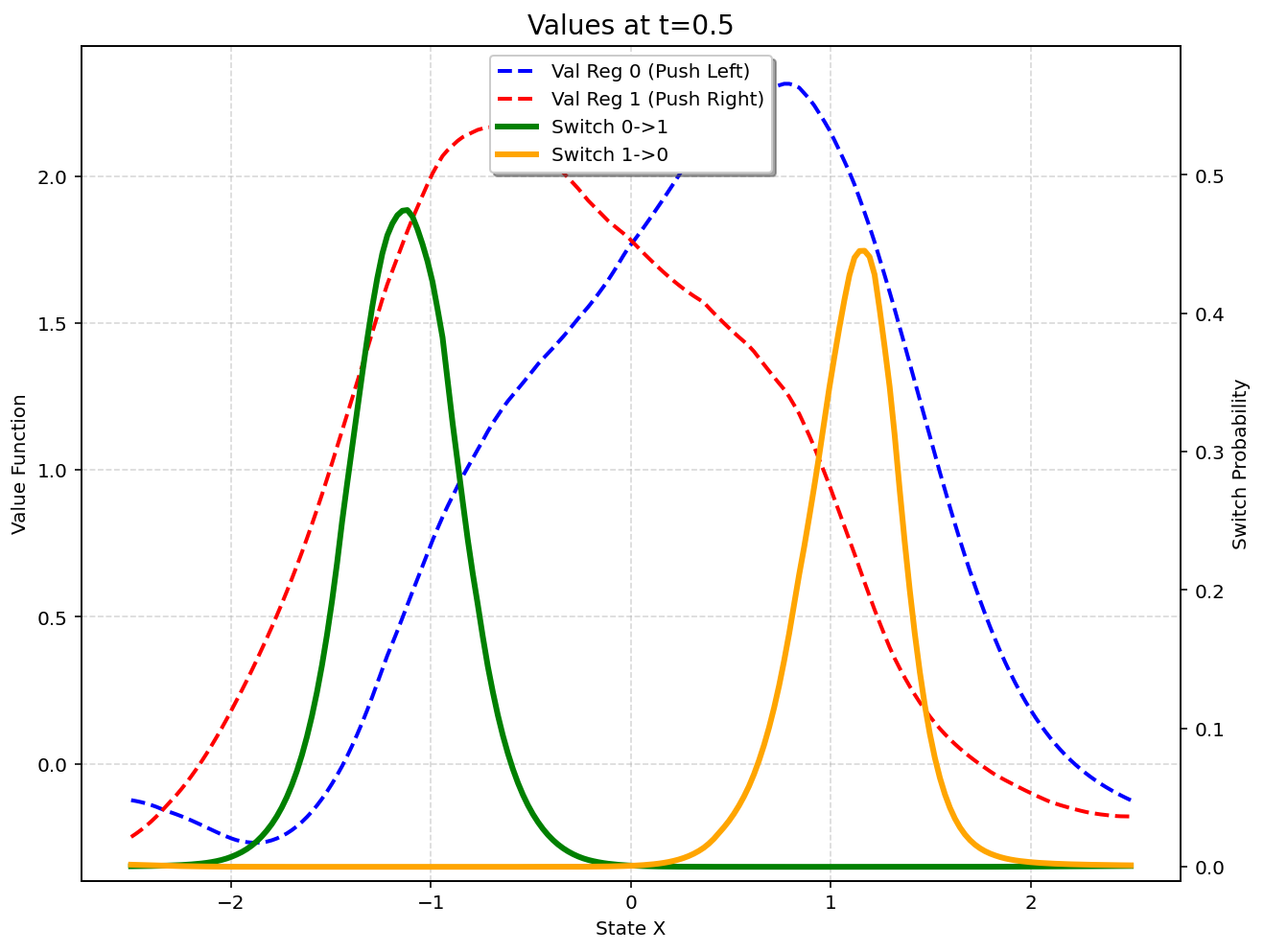}
    }
  \caption{(a): Convergence of the training loss for the bounded regulator problem with $\lambda=0.2$. (b): Learned value functions and switching probabilities at $t=0.5$ and $I_{0.5}=0$ for \( \lambda=0.2 \), i.e, $P(I_{t+\Delta t} = 1 | I_t = 0, X_t = x, t = 0.5)$.}

  \label{fig:convergence}
\end{figure}

A central theoretical result is the convergence of the exploratory solution to the classical optimal switching policy as the temperature parameter $\lambda$ tends to zero. We validate this numerically.
Figure~\ref{fig:compare_lambda_convergence}-(a) shows that the training loss decreases for different values of $\lambda$, with convergence achieved in all cases. More importantly, Figure~\ref{fig:compare_lambda_convergence}-(b) illustrates the fundamental transformation of the optimal policy. For a larger $\lambda$ (e.g., 0.2), the switching probability is a smooth function of the state, reflecting exploratory randomization. As $\lambda$ decreases to $1\times 10^{-6}$, the probability curve becomes sharp and nearly binary, approaching a deterministic threshold-based policy.

\begin{figure}[ht]
  \centering
   \subfigure[]{
        \includegraphics[width=0.45\textwidth]{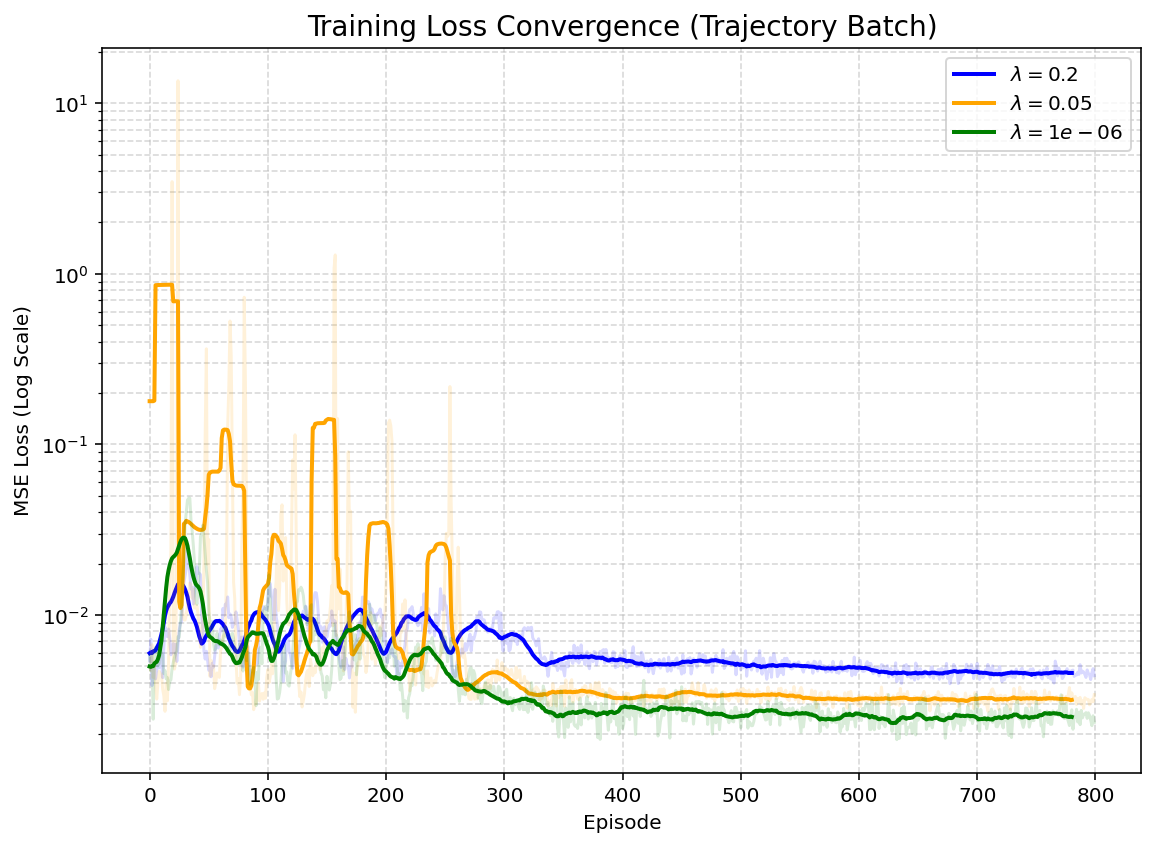}
    }
       \subfigure[]{
        \includegraphics[width=0.45\textwidth]{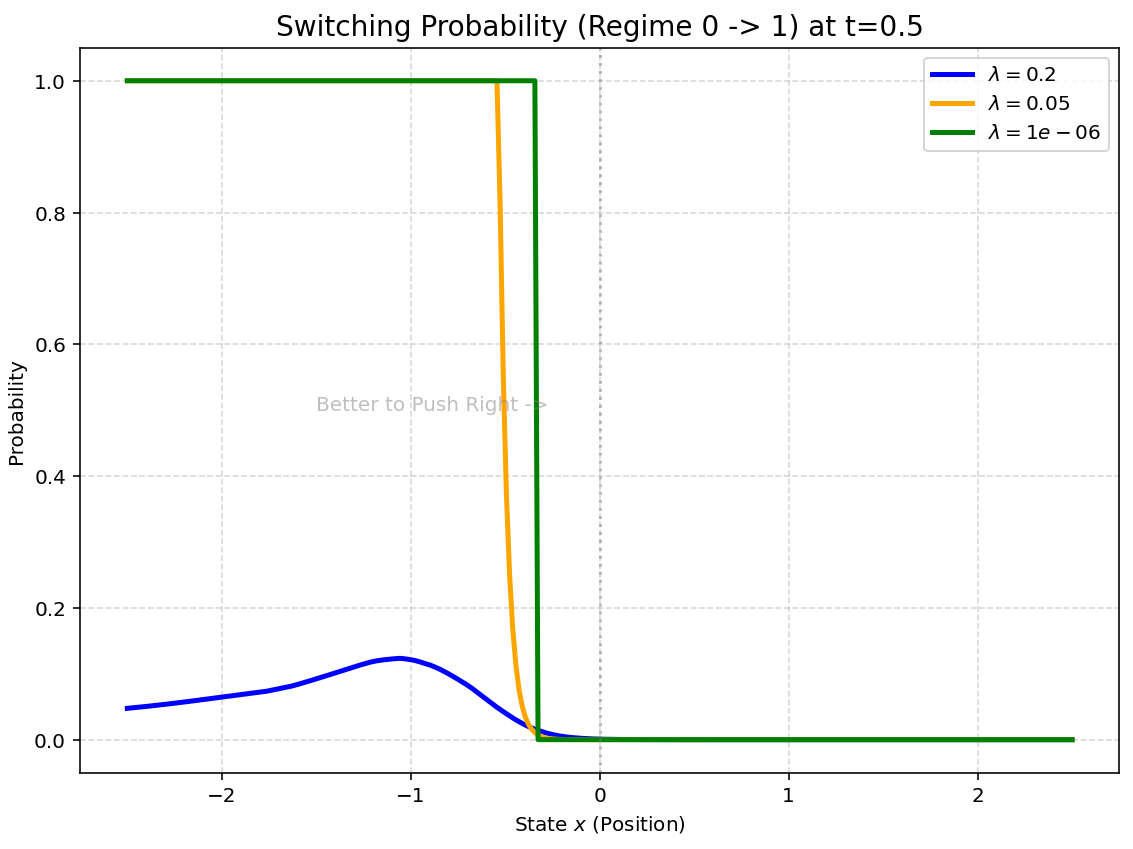}
    }
  \caption{(a): Training convergence for different temperature parameters \( \lambda \). (b): Evolution of the switching probability from regime 0 to 1 as \( \lambda \) decreases.} \label{fig:compare_lambda_convergence}
\end{figure}

To benchmark the proposed method, we construct a classical reference solution by solving the no-exploration HJB variational inequality system (i.e., the classical optimal switching problem with $\lambda = 0$) using a finite-difference scheme on a sufficiently fine grid. This provides a fully known-model baseline against which the exploratory RL algorithm can be evaluated. The RL agent is trained under entropy regularization with $\lambda = 5, 0.1, 10^{-6}$, and the resulting value functions are compared with the finite-difference reference at two representative time instants, $t = 0.5$ and $t = 0.8$, as shown in Figure~\ref{fig:benchmark-example1}. The comparison reveals that the learned RL value functions faithfully capture the overall shape of the classical solution across both regimes. As anticipated, the closest match occurs at the smallest temperature parameter $\lambda = 10^{-6}$, which lends strong numerical support to Theorem~\ref{thm:convergence-lambda}: the exploratory HJB solutions converge to the classical variational inequality solution as $\lambda \to 0$. These results thus confirm that the proposed continuous-time RL framework is capable of recovering the value structure of the classical optimal switching problem in the vanishing-regularization limit.

\begin{figure}[h]
  \centering
    \centering
   \subfigure[]{
  \includegraphics[width=0.9\textwidth]{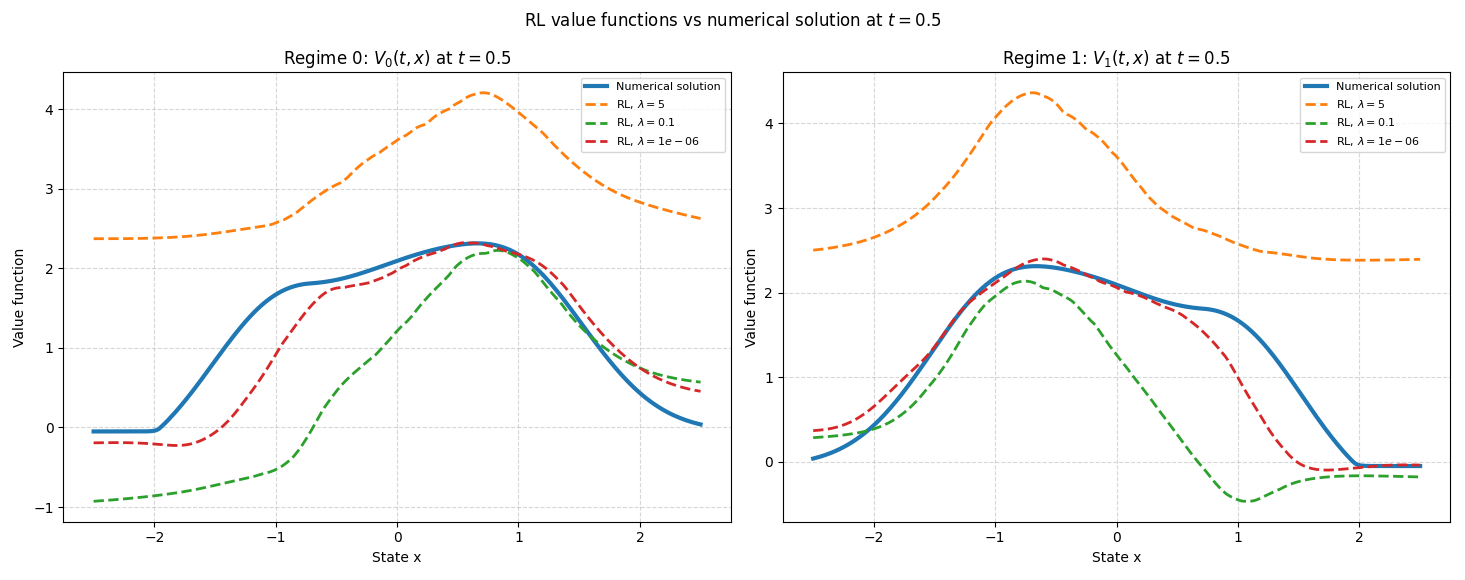}
   }

       \subfigure[]{
  \includegraphics[width=0.9\textwidth]{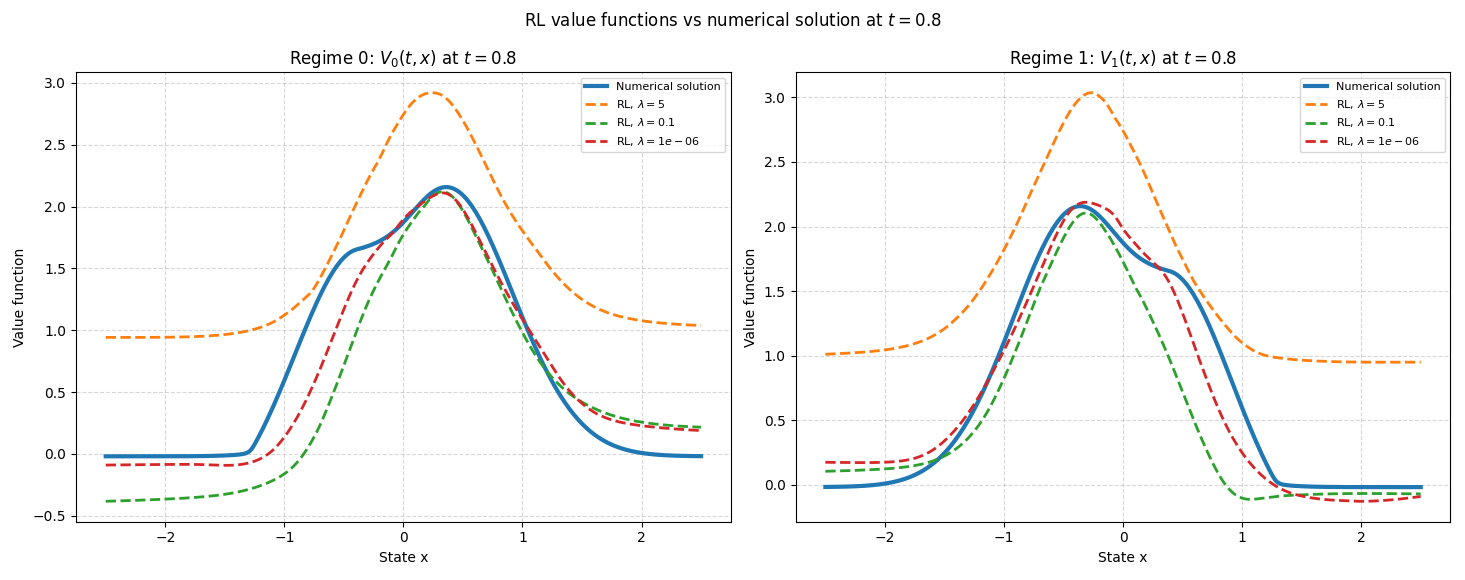}
  }
  \caption{ The learned entropy-regularized value functions under different temperature parameters $\lambda$\ \ vs\ \ the numerical value function at (a) $t=0.5$; (b) $t=0.8$.}
  \label{fig:benchmark-example1}
\end{figure}



\subsection{Put Option Selection Problem}
In this second example, we consider an investor who aims to optimally switch an investment decision between three regimes: two European put options on different assets and a risk-free savings account. The investor's wealth can be allocated to one of three regimes during the finite horizon $[0,T]$:
\begin{itemize}
    \item regime 0: a put option on Stock $A$.
    \item regime 1: a put option on Stock $B$.
    \item regime 2: the risk-free savings account.
\end{itemize}
The underlying stock prices follow the geometric Brownian motion:
\begin{align*}
dS^A_t = \mu^A S^A_t dt + \sigma^A S^A_t dW_t, \quad dS^B_t = \mu^B S^B_t dt + \sigma^B S^B_t dW_t,\quad t\in(0,T],
\end{align*}
with $S^A_0=s^A\in[0,\infty), S^B_0=s^B\in[0,\infty)$. The parameters are set as $(\mu^A,\sigma^A)=(0.1,0.2)$ and $(\mu^B,\sigma^B)=(0.05,0.1)$, and $W = (W_t)_{t \in [0, T]}$ is a standard Brownian motion.  The risk free rate is $r=0.05$.  For any time $t\in[0,T]$, the investor decides an action $I_t\in\{1,2,3\}$, which determines the regime in which the investor's wealth is allocated. Switching between regimes incurs transaction costs given by the matrix:
\begin{align*}
G = (g_{ij})_{1\leq i,j\leq 3}=
\begin{bmatrix}
0 & 0.02 & 0.01 \\
0.02 & 0 & 0.01 \\
0.02 & 0.02 & 0
\end{bmatrix}.
\end{align*}

The investor's objective is to maximize the expected total  reward over the horizon $[0, T]$, where the running reward function is given by
\begin{align*}
f(s^A,s^B,i) =
\begin{cases}
(S_K-s^A)^+, &i=1,\\
(S_K-s^B)^+, &i=2,\\
rS_K, &i=3,
\end{cases}
\end{align*}
with the strike price $S_K= 1$ and $(x)^+:=\max\{x,0\}$ for $x\in\R$. The terminal reward function is assumed to be $0$.

We set the time horizon $T=1$, the number of time intervals $K=50$, the step size $\Delta t=T/K=0.02$, and  the temperature parameter  $\lambda = 0.1$. The value function and policy are approximated by a neural network with the architecture and parameters summarized in Table~\ref{tab:nn_params_regulator-portfolio}. The model was implemented within the PyTorch framework.

\begin{table}[htbp]
\centering
\caption{Neural Network Architecture and Training Parameters for the Regulator Problem}
\label{tab:nn_params_regulator-portfolio}
\begin{tabular}{lc}
\toprule
\textbf{Component} & \textbf{Specification} \\
\midrule
Network Architecture & 2 hidden layers \\
Activation Functions & Tanh (Layer 1), Tanh (Layer 2) \\
Hidden Dimension & 128 \\
Batch Size & 1024 \\
Optimizer & Adam \\
Learning Rate & \( 1 \times 10^{-4} \) \\
Training Episodes & 1000 \\
\bottomrule
\end{tabular}
\end{table}

According to Figure \ref{fig:loss_option}, at the beginning of training, the loss exhibits a oscillation, and the convergence is very pronounced. It becomes stable when the number of episodes exceeds 800. We then plot in  Figure \ref{fig:asset_allocation} the learnt value functions $(V^1,V^2,V^3)$. Conforming to financial intuition, each value function $V^i$ (for $i=1,2,3$)  is a decreasing function of both underlying stock prices  $s^A$ and $s^B$. 
 \begin{figure}[htbp]
  \centering
  \includegraphics[width=0.6\textwidth]{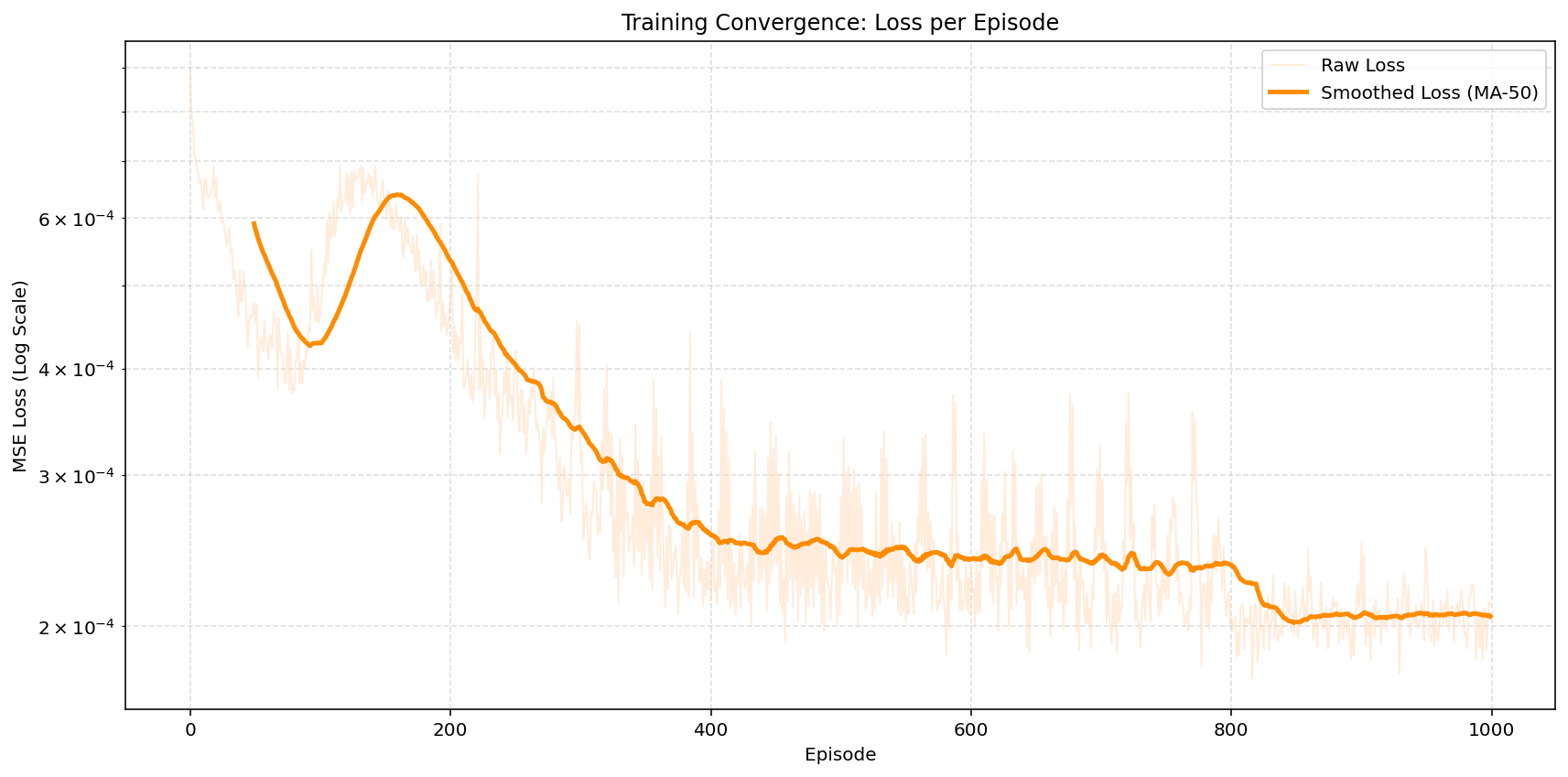}
  \caption{The training loss for the  put option selection problem.}
  \label{fig:loss_option}
\end{figure}

We further validate the algorithm by comparing its outputs with a finite-difference solution of the classical HJB variational inequality system (corresponding to $\lambda = 0$) on the same example, this time assuming full knowledge of the model coefficients. Figure~\ref{fig:benchmark-example2} displays this comparison at $t = 0.5$ along representative cross-sections of the two-dimensional state space, with one state variable held fixed to highlight the dependence on the other. The RL value functions are computed for $\lambda = 0.01, 0.5, 1$. As expected, smaller $\lambda$ yield solutions increasingly close to the no-exploration benchmark; notably, the curve for $\lambda = 0.01$ already exhibits excellent qualitative and quantitative agreement with the finite-difference reference. At the smallest tested value, the learned functions become nearly indistinguishable from the classical solution across the entire domain, offering compelling evidence of convergence as the temperature vanishes.  This numerical comparison confirms that the proposed continuous-time RL method successfully recovers the value structure of the classical optimal switching problem in the limit of vanishing entropy regularization, thereby demonstrating both the validity of the theoretical framework and the practical effectiveness of the algorithm.

%

\begin{figure}[htbp]
  \centering
  \includegraphics[width=0.9\textwidth]{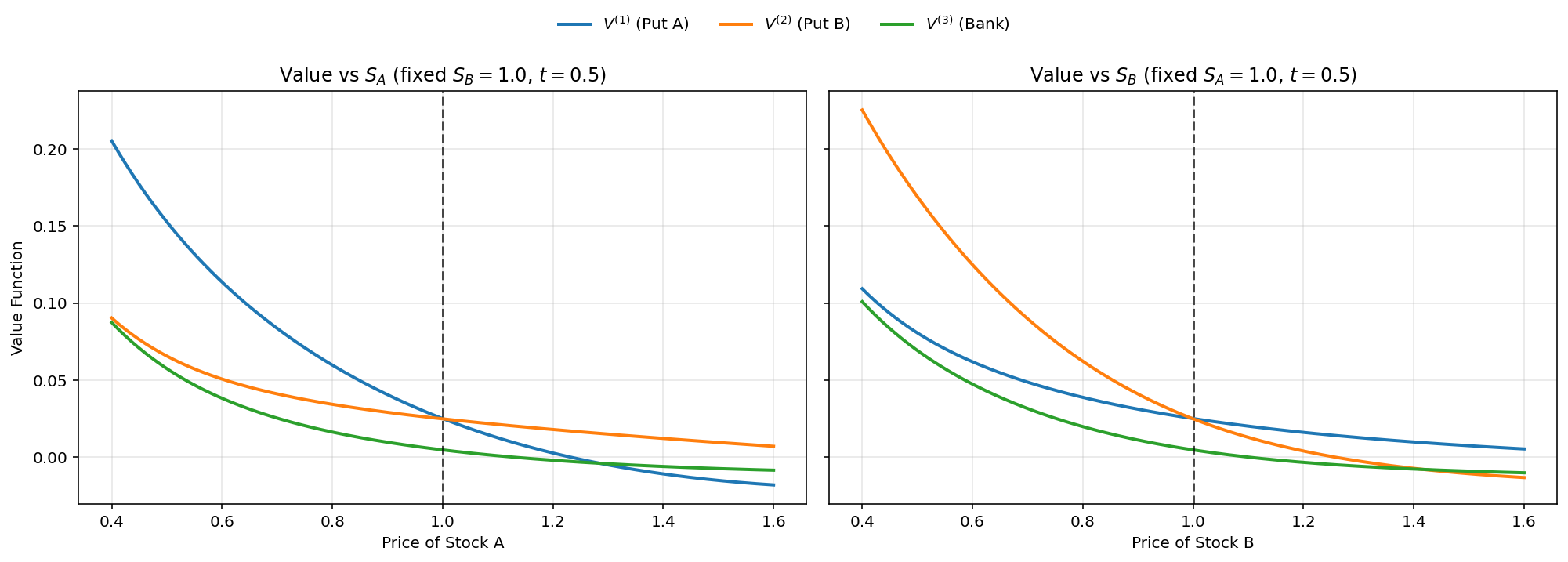}
  \caption{Left panel: The learnt value functions $s^A\to (V^1,V^2,V^3)$ with $(s^B,t)=(1.0,0.5)$. Right panel: The learnt value functions $s^B\to (V^1,V^2,V^3)$ with $(s^A,t)=(1.0,0.5)$. }
  \label{fig:asset_allocation}
\end{figure}

\begin{figure}[htbp]
  \centering
    \centering
  \includegraphics[width=0.9\textwidth]{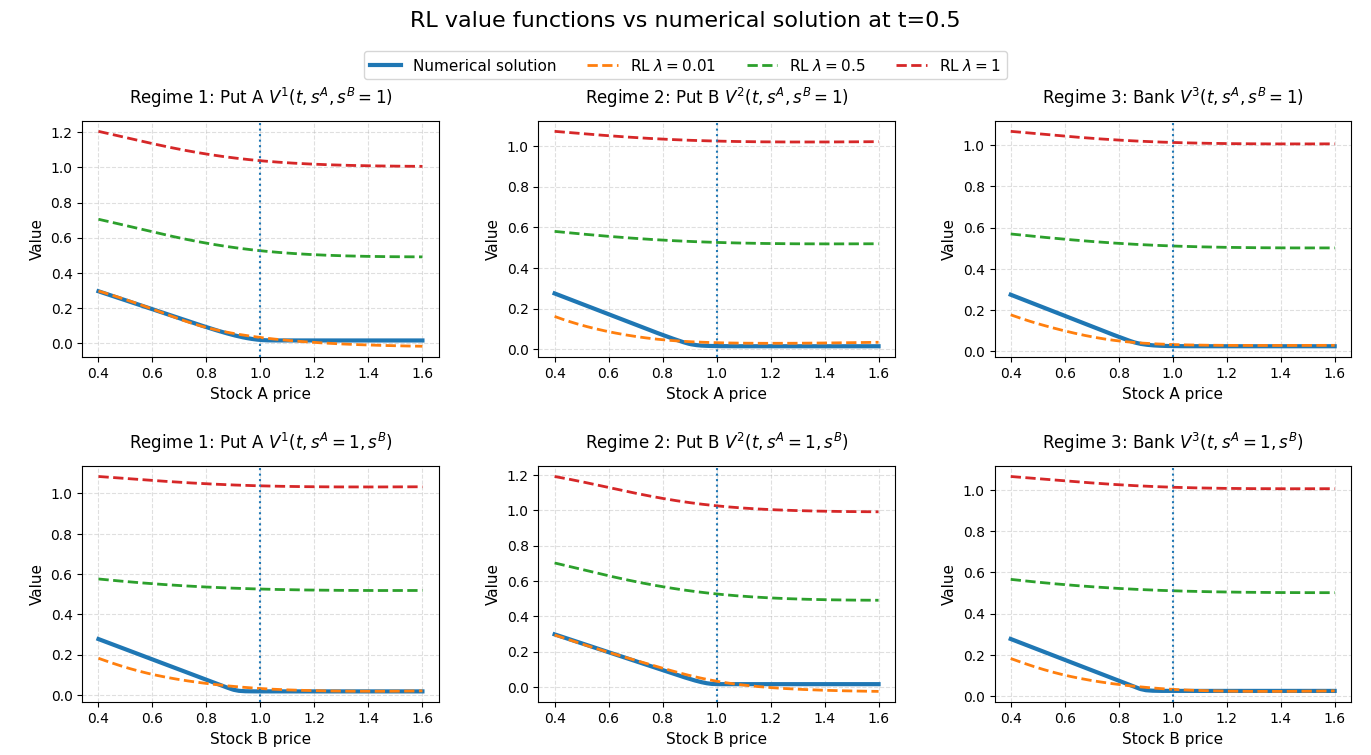}
  \caption{ The learned entropy-regularized value functions under different temperature parameters $\lambda$\ \ vs\ \ the numerical  value function at  $t=0.5$.}
  \label{fig:benchmark-example2}
\end{figure}

\newpage
\noindent
\textbf{Acknowledgements}:  Yijie Huang, Mengge Li and Xiang Yu are supported by the Hong Kong RGC General Research Fund (GRF) under grant no. 15211524 and the Hong Kong Polytechnic University research grant under no. P0045654.

\bibliographystyle{abbrvnat}
{\small
\bibliography{reference}
}
\end{document}